\documentclass[a4paper]{amsart}

\input xypic
\input xy \xyoption{all}
\usepackage{tikz, enumerate, mathtools, caption}
\usepackage{amsmath}

\usetikzlibrary{matrix}

\newtheorem{theorem}{Theorem}[section]
\newtheorem{proposition}[theorem]{Proposition}
\newtheorem{lemma}[theorem]{Lemma}

\theoremstyle{definition}

\newtheorem{remark}[theorem]{Remark}

\def\multiset#1#2{\ensuremath{\left(\kern-.3em\left(\genfrac{}{}{0pt}{}{#1}{#2}\right)\kern-.3em\right)}}
\def\supermultiset#1#2{\ensuremath{\:\left(\kern-.5em\left(\genfrac{}{}{0pt}{}{#1}{#2}\right)\kern-.5em\right)\:}}

\newcommand*{\lcdot}{\raisebox{-0.25ex}{\scalebox{4}{$\cdot$}}}

\DeclareRobustCommand{\stirling}{\genfrac\{\}{0pt}{}}

\DeclareMathOperator\lcm{lcm}
\DeclareMathOperator\im{im}

\begin{document}
\title{The cohomology of the free loop spaces of $SU(n+1)/T^n$}

\author{Matthew I. Burfitt}
\address{\scriptsize{Institute of Mathematics, University of Aberdeen, Aberdeen AB24 3UE, United Kingdom}}  
\email{matthew.burfitt@abdn.ac.uk}

\author{Jelena Grbi\'c} 
\address{\scriptsize{School of Mathematics, University of Southampton,Southampton SO17 1BJ, United Kingdom}}  
\email{J.Grbic@soton.ac.uk} 

\subjclass[2010]{}
\thanks{Research supported in part by The Leverhulme Trust Research Project Grant  RPG-2012-560.}

\maketitle

\begin{abstract}
    We study the cohomology of the free loop space of $SU(n+1)/T^n$, the simplest example of a complete flag manifolds and an important homogeneous space. Through this enhanced analysis we reveal rich new combinatorial structures arising in the cohomology algebra of the free loop spaces. We build new theory to allow for the computation of $H^*(\Lambda(SU(n+1)/T^{n});\mathbb{Z})$, a significantly more complicated structure than other known examples. In addition to our theoretical results, we explicitly implement a novel integral Gr\"obner basis procedure for computation. This procedure is applicable to any Leray-Serre spectral sequence for which the cohomology of the base space is the quotient of a finitely generated polynomial algebra. The power of this procedure is illustrated by the explicit calculation of $H^*( \Lambda(SU(4)/T^3);\mathbb{Z})$. We also provide a python library with examples of all procedures used in the paper.
\end{abstract}

\section{Introduction}

    The free loop space $\Lambda X$ of a topological space $X$ is defined to be the mapping space $Map(S^1,X)$, the space of all unpointed maps from the circle to $X$.
	This differs from the based loops space $\Omega X=Map_*(S^1,X)$, the space of all pointed maps from the circle to $X$.
	The two loop spaces are connected by the evaluation fibration.
	The based loop space functor is an important classical object in algebraic topology and has been well studied.
	However, the topology of free loop spaces, while required for many applications,
	behaves in a much more complex way
	and is still only well understood in a handful of examples.
    A primary motivation for studying the topology of the free loop space is the important role loops on a manifold play in both mathematics and physics.
    Given a Riemannian manifold $(M,g)$, the closed geodesics parametrised by $S^1$ are the critical points of the energy functional
    \begin{equation*}\label{eq:Energy}
        E\colon \Lambda M\to \mathbb R, \quad E(\gamma):= \frac{1}{2}\int_{S^1} ||\dot{\gamma}(t)||^2 dt.
    \end{equation*}
    Morse theory applied to the energy functional $E$ gives a description of the loop space $\Lambda M$ by successive attachments of bundles over the critical submanifolds.
    Knowledge of the topology of $\Lambda M$ therefore implies existence results for critical points of $E$.
    Computations of the cohomology of the free loop space have therefore received much attention over the last several decades.
    In the simplest case when $X$ is an $H$-space,
    there is a homotopy equivalence
    \[
        \Lambda X \simeq \Omega X \times X.
    \]
    Progress past this is restricted to specific examples by applying specialised methods relevant to their particular case. 
    Broadly speaking there are three main related approaches to studying the cohomology of the free loop space:
    the Hochschild cohomology of the normalized singular chains on the base loop space \cite{Goodwillie85, Burghelea86, Dupont03, Ndombo2002, Idrissi00, Menichi00},
    the Eilenberg-Moore spectral sequence of the fibre square realising $\Lambda X$ as the pullback back of a pair of diagonal maps
    \cite{Smith81, Kuribayashi91, Kuribayashi99, Kuribayashi04}
    and the cohomology Leray-Serre spectral sequence of the path-loop (or Wegraum) fibration \cite{McCleary1987, cohololgy_Lprojective, Burfitt2018}.
    Computational examples in the literature include
    the complex projective space with integral coefficients \cite{Crabb88},
    Grassman and Stiefel manifolds with coefficients in a finite field in many cases \cite{Kuribayashi91},
    wedge products of same dimensional spheres with integral coefficients \cite{Parhizgar97},
    the classifying space of a compact simply connected Lie groups with coefficients in a finite field \cite{Kuribayashi99}
    and
    simply connected $4$-manifolds with rational coefficients \cite{Onishchenko12}.
 
    The importance of the topology of the free loop space of closed oriented manifolds is further highlighted by the seminal work of Chas and Sullivan \cite{StringTopology} where two new algebraic operations, the loop product and Batalin–Vilkovisky operator where introduced, collectively referred to as string topology operations.
    In studying string topology operations the homology of the free loop space has also been considered in a several additional cases.
    Integral homology of the free loop space of complex Stiefel manifold \cite{Tamanoi07},
    spaces whose cohomology is an exterior algebra with field coefficients \cite{Bohmann2021},
    $(n-1)$-connected manifolds up to dimension $3n-2$ with homology coefficients over a field \cite{Berglund15}
    and many cases of $(n-1)$-connected $2n$-manifolds with integral coefficients \cite{Beben17}.
    Cohen-Jones-Yan \cite{Cohen2004} have also shown that there is a spectral sequence of algebras converging to the homology of the free loop space and the loop product demonstrating its use on spheres and complex proactive spaces.
    
    A straight forward Leray-Serre spectral sequence approach to computing the cohomology of the free loop space was presented by Seeliger \cite{cohololgy_Lprojective} and demonstrated on the free loop space of complex projective spaces.
    Greatly extending these ideas the authors \cite{Burfitt2018} previously obtained the integral cohomology of the free loop space of the complete flag manifolds of rank $2$ simple Lie groups.
    The other work closely related to this paper is that of McCleary and Ziller~\cite{McCleary1991, McCleary1987}, where it is shown using the Leray-Serre spectral sequence on the path-loop fibration and a classical theorems Gromoll-Meyer~\cite{Gromoll69}, that all homogeneous spaces with the exception of those of rank $1$ have infinitely many geometrically distinct closed geodesics.
    This extends earlier work of Ziller~\cite{Ziller1977}, in which Morse theory is applied to obtain the $\mathbb{Z}_2$ Betti number of the free loop space of globally symmetric spaces.

	In this paper we explore the cohomology of the free loop space of homogeneous spaces by studying $H^*(\Lambda(SU(n+1)/T^n);\mathbb{Z})$ for $n \geq 2$.
	In doing so we uncover surprising combinatorial structure, make use of computational commutative algebra and develop computer aided algorithms.
	The power of the theory is illustrated by obtaining the integral cohomology of $\Lambda(SU(4)/T^3)$, a significantly more complex example than obtaining $H^*(\Lambda S^3;\mathbb{Z})$ in the case when $n=2$ or $H^*(\Lambda(SU(3)/T^2);\mathbb{Z})$ in the case the $n=3$ previously considered in \cite{Burfitt2018}.
    
    We apply classical homotopy theoretic arguments to the path-loop
    fibration and its pull back along the diagonal map on $SU(n+1)/T^n$ to derive the differentials in the Serre spectral sequence of the free loop space evaluation fibration converging to $H^*(\Lambda(SU(n+1)/T^n);\mathbb{Z})$.
    Classically, the elementary symmetric polynomials are used for the basis of symmetric functions in order to write down generators of the quotient ideal in $H^*(SU(n+1)/T^n;\mathbb{Z})$.
    However, that choice of the basis elements does not lead to a description of the differentials that can be easily applied to developing further theory.
    In this work we choose the basis consisting of complete homogeneous symmetric polynomials.
    By doing so we acquire a new unexpectedly sophisticated combinatorial structure on the differentials.

    The consequences of the choice of basis is first highlighted by Theorem~\ref{thm:monomial sum} where the ideal generated by complete homogeneous symmetric generator is shown to straightforwardly rearrange to a reduced Gr\"obner basis.
    This demonstrates our new approach of using Gr\"obner basis for understating cohomology algebras expressed as a polynomial quotients by analysing them from the perspective of computational commutative algebra and computer aided algorithms. 
    We explicitly apply Gr\"obner basis to spectral sequences in Proposition~\ref{thm:SpectralGrobner} and lay out an integral Gr\"obner basis procedure for performing computations applicable to any Leray-Serre spectral sequence for which the cohomology of the base space is the quotient of a finitely generated polynomial algebra.
    The enhancements to the classical Buchberger algorithm in Section~\ref{sec:SpectralGrobner} combine to provided a powerful procedure that makes the later application in Section~\ref{sec:LSU4/T3} computationally possible.

    It is the characterisation in Proposition~\ref{thm:SpectralGrobner} that motivates Theorem~\ref{thm:Ideals}, which provides part of the computation of $H^*(\Lambda(SU(n+1)/T^n);\mathbb{Z})$ for arbitrary $n$.
    Theorem~\ref{thm:Ideals} is used along side the direct application of Proposition~\ref{thm:SpectralGrobner} in Section~\ref{sec:LSU4/T3} to obtain an expression for the module structure of $H^*(\Lambda(SU(4)/T^3);\mathbb{Z})$ up to a small uncertainty in torsion type.

\section{Background}\label{sec:Background}
    
    \subsection{Symmetric polynomials}\label{sec:SymPoly}

    	A polynomial in $\mathbb{Z}[\gamma_1,\dots,\gamma_n]$ is called symmetric if it is invariant under permutations of the indices of variables $\gamma_1,\dots,\gamma_n$.
    	The study of symmetric polynomials goes back more than three hundred years, originally used in the study of roots of single variable polynomials.
    	Today symmetric polynomials have applications in a diverse range of areas of mathematics.
    	In the paper the relevance of the symmetric polynomials is brought by their presence in the cohomology rings of complete flag manifolds,
    	in Section~\ref{sec:CohomCompFlag}.
    	In this section we summaries some basic concepts from the theory of symmetric polynomials that will be essential for our later work.
    	A compete introduction to the topic can be found in \cite[\S $7$]{ECstanly} or \cite[\S $I$]{Macdonald}.
	
    	    \subsubsection{Elementary symmetric polynomials}\label{sec:elementary}
    	
    		Much of the language used to described symmetric polynomials is the language of partitions.
    			An $n$ \emph{partition} $\lambda$ is a sequence of non-negative integers $(\lambda_1,\dots,\lambda_k)$, for some integer $k\geq 1$, such that
    			\begin{equation*}
    				\lambda_1\geq\cdots\geq\lambda_k \;\; \text{and} \;\; \lambda_1+\cdots+\lambda_k=n.
    			\end{equation*}
    			By convention we consider partition $(\lambda_1,\dots,\lambda_k)$ and $(\lambda_1,\dots,\lambda_k,0,\dots,0)$ to be equal
    			and abbreviate an $n$ partition $\lambda$ by $\lambda \vdash n$.
    		The elementary symmetric polynomials are a special collection of symmetric polynomials that form a basis of the symmetric polynomials, which we make explicit in Theorem~\ref{thm:FunThmSym}.
    			For each integer $n\geq 1$ and $1\leq l \leq n$, the \emph{elementary symmetric polynomial} $\sigma_l\in\mathbb{Z}[\gamma_1,\dots,\gamma_n]$ in $n$ variables is given by
    			\begin{equation*}
    				\sigma_l=\sum_{1\leq i_1<\cdots<i_l\leq n}{\gamma_{i_1}\cdots\gamma_{i_l}}.
    			\end{equation*}
    		For a partition $\lambda=(\lambda_1,\dots,\lambda_k)$, denote by $\sigma_\lambda$ the symmetric polynomial $\sigma_{\lambda_1}\cdots\sigma_{\lambda_k}$.
    		The following theorem is sometimes known as the fundamental theorem of symmetric polynomials.
    		
    		\begin{theorem}[{\cite[\S $7.4$]{ECstanly}}]\label{thm:FunThmSym}
    			For each $n\geq 1$, the set of $\sigma_\lambda$, where $\lambda$ ranges over all $n$ partitions, forms an additive basis of all symmetric polynomials.
    			That is, for $1\leq i \leq n$, the set of $\sigma_i$ are independent and generate 
    		    the symmetric polynomials as an algebra.
    		\end{theorem}
            \hspace*{\fill} $\square$
 	
	    \subsubsection{Complete homogeneous symmetric polynomials}\label{sec:homogeneous}
	
    		The complete homogeneous symmetric polynomials are another collection of $n$ symmetric polynomials in $n$ variables for each $n\geq 1$.
    		In a sense, which is made explicit in \cite[\S $7.6$]{ECstanly},
    		the complete homogeneous symmetric polynomials can be thought of as dual to the elementary symmetric polynomials.
    			For each integer $n\geq 1$ and $1\leq l \leq n$, define the \emph{complete homogeneous symmetric polynomials} $h_l\in\mathbb{Z}[\gamma_1,\dots,\gamma_n]$ in $n$ variables by
    			\begin{equation}\label{defn:CompleteHomogeneous}
    				h_l=\sum_{1\leq i_1\leq\cdots\leq i_l\leq n}{\gamma_{i_1}\cdots\gamma_{i_l}}.
    			\end{equation}
    			For a partition $\lambda=(\lambda_1,\dots,\lambda_k)$, denote by $h_\lambda$ the symmetric polynomial $h_{\lambda_1}\cdots h_{\lambda_k}$.
    		Setting $\sigma_0=h_0=1$, the following identity derived for infinite variables in \cite[\S $I$,2]{Macdonald}, gives the relationship between the elementary symmetric and complete homogeneous symmetric polynomials.
    		Evaluating all but $n$ variables to $0$, for any $m \leq n$ gives
    		\begin{equation}\label{eq:ElementaryHomogeneousRelation}
    		    \sum_{t=0}^{m}(-1)^t \sigma_t h_{n-t} = 0.
    		\end{equation}
    		As $\sigma_0=h_0=1$, equation~(\ref{eq:ElementaryHomogeneousRelation}) can be used to inductively derive expressions for either elementally symmetric polynomials in terms of complete homogeneous polynomials or the other way round.
    		Moreover the complete homogeneous symmetric polynomials
    		also form a basis of the symmetric polynomials.
    
    \subsection{Gr\"{o}bner bases}\label{sec:Grobner}	
		
		Gr\"{o}bner basis provide a powerful tool to perform computations on ideals in commutative algebra.
		Their use however extends far beyond such calculations having applications within mathematics, computer science physics and engineering.
		We now briefly describe the Gr\"{o}bner basis theory used later in the paper, for further details see \cite{Grobner2} or \cite{Grobner1}.
		The following results are stated over a Euclidean or principal ideal domain $R$.
		However throughout this paper we consider only the case when $R=\mathbb{Z}$.
		Gr\"{o}bner basis theory can be generalised to other rings and stronger results can be recovered over a field.
		
		Given a finite subset $A$ of $R[x_1,\dots,x_n]$, we denote by $\langle A\rangle$ the ideal generated by elements of $A$.
		Form now on we assume a total monomial ordering on the polynomial ring $R[x_1,\dots,x_n]$ which respects multiplication.
		In the course of this paper we take this order to be the lexicographic ordering.
		The \emph{leading term} of a polynomial with respect to an order is the term largest with respect to the order, the \emph{leading monomial} is the leading term multiplied by its coefficient the \emph{leading coefficient}.
    		For $f,g,p\in R[x_1,\dots,x_n]$, the polynomial $g$ is said to be {\it reduced} from $f$ by $p$, written
    		\begin{equation*}
    		    f \xrightarrow{p} g
    		\end{equation*}
    		if there exists a monomial $m$ in $f$ such that the leading monomial $l_p$ of $p$ divides $m$, say $m= m'l_p$ for some monomial $m'\in R[x_1,\dots,x_n]$ and $g=f-m'p$.
		    Let $R$ be a principal ideal domain and let $G$ be a finite subset of $R[x_1,\dots,x_n]$.
		    Then $G$ is a {\it Gr\"{o}bner basis} if any of the following equivalent conditions hold.
		    \begin{enumerate}
    		    \item
    		    The ideal of leading terms of $\langle G\rangle$ is equal to the ideal generated by the leading terms of $G$.
    		    \item
    		    All elements of $\langle G\rangle$ can be reduced to zero by elements of $G$.
    		    \item
    		    Leading terms of elements in $\langle G\rangle$ are divisible by a leading terms of an elements in $G$.
		    \end{enumerate}
		A set is called {\it decidable} if for any two elements input, there is an algorithm that can determine whether they are equal. 
		A ring is called {\it computable} if it is decidable as a set and there is an effectively computable algorithm for addition, multiplication and subtraction in the ring for an input of a pair of elements.
		A principal ideal domain is called a {\it computable principal ideal domain} if it is a computable ring, there is an algorithm that can effectively compute whether a given pair of elements is divisible and an extended Euclidean algorithm can be effectively computed. 
		Euclidean domain is a {\it computable Euclidean domain} if it is a computable ring and there is an algorithm that effectively computes division with remainder.
		The integers are a computable Euclidean domain.
		Moreover, as division with remainder can be applied to construct an extended Euclidean algorithm, so every computable Euclidean domain is also a computable principle ideal domain.
		
		If $R$ is a computable principle ideal domain, then for any ideal in $R[x_1,\dots,x_n]$, there exists a Gr\"{o}bner basis.
		In particular, for finite $A\subseteq R[x_1,\dots,x_n]$ there is an algorithm
		to obtain a Gr\"{o}bner basis $G$ such that $\langle G\rangle=\langle A\rangle$. 
		Over a field the most efficient algorithm is known as the Buchberger algorithm and can easily be implemented by a computer and a similar algorithm can used for principal ideal domains.
		Over a Euclidean domain computation speed might be improved with the implementing of more advanced algorithms \cite{Lichtblau13, Eder19}.
        A basic Gr\"{o}bner basis algorithm can be deduced from the following theorem.
        
		    Let $g_1,g_2\in R[x_1,\dots,x_n]$ be non-zero with leading terms $t_1,t_2$ and leading coefficients $c_1,c_2$.
		    Set $b_1,b_2\in R$ be such that $b_1c_1=b_2c_2=\lcm(c_1,c_2)$ and $s_1,s_2\in R[x_1,\dots,x_n]$ be such that $s_1t_1=s_2t_2=\lcm(t_1,t_2)$.
		    Then the $S$-polynomial of $g_1$ and $g_2$ is given by
		    \begin{equation*}
		        Spol(g_1,g_2) = b_1s_1g_1 - b_2s_2g_2.
		    \end{equation*}
		    Set $d_1,d_2\in R$ to be $d_1c_1=d_2c_2=\gcd(c_1,c_2)$.
		    Then $G$-polynomial of $g_1$ and $g_2$ is given by
		    \begin{equation*}
		        Gpol(g_1,g_2) = d_1s_1g_1 + d_2s_2g_2.
		    \end{equation*}

		\begin{theorem}[\cite{Grobner1}]\label{thm:GrobnerAlg}
		      A finite subset $G$ of $R[x_1,\dots,x_n]$ is a Gr\"{o}bner basis if for any $g_1,g_2\in G$
		      \begin{enumerate}
		              \item
		              $Spol(g_1,g_2)$ reduces to $0$ by $G$ and
		              \item
		              $Gpol(g_1,g_2)$ is reducible by an element of $G$ in its leading term.
		      \end{enumerate}
		\end{theorem}
	    \hspace*{\fill} $\square$
	    
		In the computational part of our work it is important that a Gr\"{o}bner basis can be used to compute the intersection of ideals. This procedure is made explicit in the next remark.
		
		\begin{remark}\label{rmk:intersectionGrobner}
		    Let $A=\{a_1,\dots,a_s\}$ and $B=\{b_1,\dots,b_l\}$ be subsets of $R[x_1,\dots,x_n]$.
		    Take a Gr\"{o}bner basis $G$ of 
		    \begin{equation*}
		        \{ya_1,\dots,ya_t,(1-y)b_1,\dots,(1-y)b_l\}
		    \end{equation*}
		    in $R[x_1,\dots,x_n,y]$ using a monomial ordering in which monomials containing $y$ are larger than $y$ free monomials.
		    Then a Gr\"{o}bner basis of $\langle A\rangle\cap\langle B\rangle$ is given by the elements of $G$ that do not contain $y$.
		\end{remark}

        A Gr\"{o}bner basis over a principle ideal domain however is not unique.
		Importantly, different choices of ordering produce very different Gr\"{o}bner basis even when the basis is given in a reduced form.
		From now on let $E$ be a Euclidean domain.
		In this case, there may be a choice of remainder upon division resulting in alternative division algorithms and this might change the output of a Gr\"{o}bner basis algorithm. 
		However, assume the division algorithm is fixed.
		
		To find a Gr\"{o}bner basis which is unique in $E[x_1,\dots,x_n]$ we are required to be more precise about reduction.
		    For $f,g,p\in E[x_1,\dots,x_n]$, the polynomial $g$ is said to be \emph{E-reduced} from $f$ by $p$ to $g$,
    		if there exists a monomial $m=at$ in $f$ with
    		the leading term $l_p$ of $p$ dividing $t$ such that $t=sl_p$ and
    		\begin{equation*}
    		    g=f-qsp
    		\end{equation*}
    		for some non-zero $q\in E$ the quotient of $a$ upon division with unique remainder by $l_p$.
		    A Gr\"{o}bner basis basis $G$ in $E[x_1,\dots,x_n]$ is said to be \emph{reduced} if all polynomials in $G$ cannot be $E$-reduced by any other polynomial in $G$.
		
		\begin{theorem}[\cite{Kandri-Rody88}]
		    A reduced Gr\"{o}bner basis $G$ over $\mathbb{Z}[x_1,\dots,x_n]$ for which all leading monomials have positive coefficients is unique.
		\end{theorem}
		\hspace*{\fill} $\square$
		
		In general, with coefficients in a Euclidean domain the uniqueness of a reduced Gr\"{o}bner bias holds up to multiplication by a units. However, in this paper we only consider Gr\"{o}bner basis of integer polynomials.
		The following theorem expands upon part (2) of the equivalent Gr\"{o}bner basis definitions above.
		
		\begin{theorem}[\cite{Grobner1}]\label{thm:GrobnerOver}
		    Let $G$ be a Gr\"{o}bner basis in $E[x_1,\dots,x_n]$.
		    Then all elements in $E[x_1,\dots,x_n]$ E-reduce by elements of $G$ to a unique representative in
		    \begin{equation*}
    		    \frac{E[x_1,\dots,x_n]}{\langle G\rangle}.
		    \end{equation*}
		\end{theorem}
		\hspace*{\fill} $\square$
    
        \subsection{Cohomology of complete flag manifolds}\label{sec:CohomCompFlag}

		A manifold $M$ is called homogeneous if it can be equipped with a transitive $G$ action for some Lie groups $G$.
		In this case, $M \cong G/H$ for some Lie subgroup $H$ of $G$ isomorphic to the orbit of a point in $M$.
		A Lie subgroup $T$ of a Lie group $G$ isomorphic to a torus is called maximal
		if any Lie subgroup also isomorphic to a torus containing $T$ coincidences with $T$.
		The next proposition is straightforward to show, see for example \cite[\S $5.3$]{MT2} Theorem $3.15$.
		
		\begin{proposition}\label{prop:tori}
				The conjugate of a torus in $G$ is a torus and all maximal tori are conjugate.
				In addition given a maximal torus $T$, for all $x\in G$ there exists an element $g\in G$ such that $g^{-1}xg\in T$.
				Hence the union of all maximal tori is $G$.
		\end{proposition}
		\hspace*{\fill} $\square$
		
		It is therefore unambiguous to refer to the maximal torus $T$ of $G$ and consider the quotient $G/T$, which is isomorphic regardless of the choice of $T$.
		The homogeneous space $G/T$ is called the {\it complete flag manifold} of $G$.
		The rank of Lie group $G$ is the dimension of a maximal torus $T$.
		The ranks of classical simple Lie groups can be deduced by considering the standard maximal tori,
		see for example~\cite{MT}.
	
	    Borel \cite{Borel} studied in detail the cohomology of homogeneous spaces,
	    in particular deducing the rational cohomology of $G/T$.
		Following later work of Bott, Samelson, Toda, Watanabe and others the integral cohomology of complete flag manifolds of all simple Lie groups were deduced.
		The integral cohomology of the complete flag manifolds of the special unitary groups is as follows.
	
		\begin{theorem}[\cite{Borel}, \cite{AplicationsOfMorse}]\label{thm:H*SU/T}
			For each integer $n\geq 0$, the cohomology of the complete flag manifold of the simple Lie group $SU(n+1)$ is given by
			\begin{equation*}
				H^*(SU(n+1)/T^n;\mathbb{Z})=\frac{\mathbb{Z}[\gamma_1,\dots,\gamma_{n+1}]}{\langle\sigma_1,\dots,\sigma_{n+1}\rangle}
			\end{equation*}
			where $|\gamma_i|=2$.
		\end{theorem}
        \hspace*{\fill} $\square$
    
    \subsection{Based loop space cohomology of $SU(n)$}\label{sec:LoopLie}
		
		The Hopf algebras of the based loop space of Lie groups were studied by Bott in \cite{bott1958}.
		More recently, Grbi{\'c} and Terzi{\'c} \cite{homology_Lflags} showed that the integral homology of the based loop space of a complete flag manifold is torsion free
	    and found the integral Pontrjagin homology algebras of the complete flag manifolds of compact connected simple Lie groups
	    $SU(n)$, $Sp(n)$, $SO(n)$, $G_2$, $F_4$ and $E_6$.
		Recall that the integral divided polynomial algebra on variables $x_1,\dots,x_n$ is given by
		\begin{equation*}
			\Gamma_{\mathbb{Z}}[x_1,\dots,x_n]=\frac{\mathbb{Z}[(x_i)_1,(x_i)_2,\dots]}{\langle(x_i)_k-k!x_i^k\rangle}
		\end{equation*}
		where $1\leq i \leq n$, $k\geq 1$ and $x_i=(x_i)_1$.
		The next theorem is obtained
		using a Leray-Serre spectral sequence argument applied to the path space fibrations 
		\begin{equation*}
		    \Omega SU(n) \to PSU(n) \to SU(n).
		\end{equation*}
		
		\begin{theorem}\label{thm:LoopSU(n)}
			For each $n\geq 1$, the cohomology of the based loop space of the classical simple Lie group $SU(n)$ is given by
			\begin{equation*}
				H^*(\Omega(SU(n));\mathbb{Z})=\Gamma_{\mathbb{Z}}[x_2,x_4,\dots,x_{2n-2}]
			\end{equation*}
			where $|x_i|=i$ for $i=2,4,\dots,2n-2$.
		\end{theorem}
	    \hspace*{\fill} $\square$

\section{Combintorial coefficients}\label{sec:Comintorial}
		
		Before studying the the cohomology of the free loop space of $SU(n+1)/T^n$
		we first analyse some of the combinatorial structures that appear in the cohomology algebras.
		
		\subsection{Binomial coefficients}\label{subsec:Binomial} 
		
			The binomial coefficients $\binom{n}{k}$ are defined to be the number of size $k$ subsets of a size $n$ set
			and they satisfy the recurrence relation
			    $\binom{n}{k}=\binom{n-1}{k}+\binom{n-1}{k-1}$.
			It is easily shown by induction on $n$ that for $0\leq k \leq n$, $\binom{n}{k}=\frac{n!}{k!(n-k)!}$ and is zero otherwise.
			The binomial coefficients also satisfy the well known formulas
			\begin{equation}\label{eq:binom}
				\sum_{k=0}^n{\binom{n}{k}}=2^n \;\; \text{and} \;\; \sum^{n}_{k=0}{(-1)^k \binom{n}{k}}=0.
			\end{equation}
			
		\subsection{Multinomial coefficients}
		
		    Throughout this paper for integers $k\geq 2$ and $n,a_1,\dots,a_k \geq 0$, we set
		    \begin{equation}\label{eq:Multinomial}
		        \binom{n}{a_1,\dots,a_k}=\frac{n!}{a_1!\cdots a_k!}.
		    \end{equation}
		    The following expansion in term of binomial coefficients is easily verified,
		    \begin{equation*}\label{eq:MultinomilExpansionOriginal}
		        \binom{n}{a_1,\dots,a_k}=\binom{n}{a_1}\binom{n-a_1}{a_2}\cdots\binom{n-a_1-\cdots- a_{k-1}}{a_k}.
		    \end{equation*}
		    In particular
		    \begin{equation}\label{eq:MultinomilExpansion}
		        \binom{n}{a_1,\dots,a_k}=\binom{n}{a_2,\dots,a_{k-1}}\binom{n-a_1-\cdots-a_{k-1}}{a_k}.
		    \end{equation}
		        When $a_1+\cdots+a_k=n$, the expressions in equation~(\ref{eq:Multinomial}) are called the the \emph{multinomial coefficients}.
		        In this case the combinatorial interpretation of the coefficients is as the number of way to partition a size $n$ set into subsets of sizes $a_1,\dots,a_k$.
		    
		\subsection{Multiset coefficients}
		
				The number of size $k$ multisets that can be formed from elements of a size $n$ set is denoted by $\multiset{n}{k}$ and these numbers are called the {\it multiset coefficients}.
			It is well know that $\multiset{n}{k}=\binom{n+k-1}{k}$,
			hence $\multiset{n}{k}=\multiset{n-1}{k}+\multiset{n}{k-1}$.
			To the best of our knowledge the identity in the next lemma has not been shown before.
			\begin{lemma}\label{lem:combino}
				For integers $n,m\geq 1$,
				\begin{equation*}
						\sum^n_{k=0}{(-1)^k\binom{n}{k}\multiset{n}{m-k}}=0.
				\end{equation*}
			\end{lemma}
			\begin{proof}
				We prove the statement by induction on $n$.
				When $n=1$, 
				\begin{equation*}
						\sum^n_{k=0}{(-1)^k\binom{n}{k}\multiset{n}{m-k}}
						=\binom{1}{0}\multiset{1}{m}-\binom{1}{1}\multiset{1}{m-1}
						=\binom{m}{m}-\binom{m-1}{m-1}=0.
				\end{equation*}
				Suppose the lemma holds for $n=t-1\geq1$, then
				\begin{align*}
						&\sum^t_{k=0}{(-1)^k\binom{t}{k}\multiset{t}{m-k}}
						=\sum^t_{k=0}{(-1)^k\bigg(\binom{t-1}{k}+\binom{t-1}{k-1}\bigg)\multiset{t}{m-k}} \\
						&=\sum^t_{k=0}{(-1)^k\bigg(\binom{t-1}{k-1}\multiset{t}{m-k}+\binom{t-1}{k}\multiset{t-1}{m-k}+\binom{t-1}{k}\multiset{t}{m-k-1}\bigg)}=0
				\end{align*}
				as the middle term sum $\sum^{t-1}_{k=0}{\binom{t-1}{k}\multiset{t-1}{m-k}}=0$ by assumption and all other terms cancel except for $\binom{t-1}{-1}\multiset{t}{m}$, $\binom{t-1}{t}\multiset{t-1}{m-t}$ and $\binom{t-1}{t}\multiset{t}{m-t-1}$
				all of which are zero.
		\end{proof}
    
    \subsection{Stirling numbers of the second kind}\label{subsec:Stirling}
    
        Along with binomial, multinomial and multiset coefficients, Sterling numbers of the second kind appear as part of the so called \emph{12-fold way}, a class of enumerative problems concerned with the combinations of balls in boxes problems.
        The \emph{Stirling numbers of the second kind} $\stirling{n}{m}$, denote the number of ways to partition an $n$ element set into $m$ non-empty subsets.
        The Stirling numbers of the second kind satisfy a recurrence relations 
        \begin{equation}\label{eq:StirlingRecurence}
            \stirling{n}{m} = \stirling{n-1}{m-1} + m \stirling{n-1}{m}
        \end{equation}
        for integers $n\geq m \geq 1$.
        From the combinatoral definitions, we obtain the relationship between Stirling numbers of the second kind and multinomial coefficients, given by
        \begin{equation}\label{eq:StirlingExpansion}
            m!\stirling{n}{m}=
            \sum_{\substack{a_1,\dots,a_m \geq 1 \\ a_1+\cdots+a_m=n }}
            \binom{n}{a_1,\dots,a_m}.
        \end{equation}
        To the best of our knowledge the identity in the next lemma has not been shown before.
        \begin{lemma}\label{lem:StirlingIdentity}
            For each $n \geq 1$,
            \begin{equation*}
                \sum^n_{m=1}{(-1)^m m! \stirling{n}{m}} = (-1)^n.
            \end{equation*}
        \end{lemma}
        \begin{proof}
            The formula is easily seen to hold for the case $n=1$.
            Using the recurrence relation in equation~(\ref{eq:StirlingRecurence}) and induction on $n$,
            \begin{align*}
                \sum^n_{m=1}{(-1)^m m! \stirling{n}{m}}
                =& \sum^n_{m=1}{(-1)^m m! \left( \stirling{n-1}{m-1} + m \stirling{n-1}{m} \right) } \\
                =&\sum^n_{m=1}{(-1)^m m! \stirling{n-1}{m-1}} + \sum^n_{m=1}{(-1)^m m!m \stirling{n-1}{m}} \\
                =&\sum^{n-1}_{m=1}{(-1)^{m+1} (m+1)! \stirling{n-1}{m}} + \sum^{n-1}_{m=1}{(-1)^m m!m \stirling{n-1}{m}} \\
                =&\sum^{n-1}_{m=1}{(-1)^{m+1} m!(m+1) \stirling{n-1}{m}} + \sum^{n-1}_{m=1}{(-1)^m m!m \stirling{n-1}{m}} \\
                =&\sum^n_{m=1}{(-1)^{m+1}\left( m!(m+1)-m!m \right) \stirling{n-1}{m}} \\
                =&\sum^n_{m=1}{(-1)^{m+1} m! \stirling{n-1}{m}}
                =(-1)^{n}.
            \end{align*}
        \end{proof}
        By an expansion using the recurrence relation in equation~(\ref{eq:StirlingRecurence}) and Lemma~\ref{lem:StirlingIdentity},
        the well known relation $\sum^n_{m=1}{(-1)^m (m-1)! \stirling{n}{m}} = 0$ can be easily derived.
    
    \section{Alternative forms of the symmetric ideal}\label{sec:IdeaForms}
    
		Replacing $\sigma_i$ with the complete homogeneous symmetric polynomials $h_i$ as generators of the symmetric ideal, leads to a simplification of the generator expressions, practical for working with $H^*(SU(n+1)/T^n)$ as demonstrated in the next section.
			
		For each integer $n\geq 1$ and all integers $1 \leq k' \leq k \leq n$,
		define $\Phi(k,k')$ to be the sum of all monomials in $\mathbb{Z}[x_1,\dots,x_n]$ of degree $k$ in variables $x_1,\dots,x_{n-k'+1}$.
			
		\begin{theorem}\label{thm:monomial sum}
			In the ring $\frac{\mathbb{Z}[x_1,\dots,x_n]}{\langle h_1,\dots,h_n\rangle}$
			for each $1 \leq k' \leq k \leq n$, $\Phi(k,k')=0$.
			In addition 
			\begin{equation}\label{eq:OrigIdeal}
				\langle h_1,\dots,h_n\rangle=\langle\Phi(1,1),\dots,\Phi(n,n)\rangle.
			\end{equation}
			Moreover these new ideal generators form a reduced Gr\"obner bases for the ideal of $n$ variable symmetric polynomials with respect to the lexicographic term order on variables $x_1<\cdots<x_n$.
		\end{theorem}
		\begin{proof}
			Note that by definition, $h_k=\Phi(k,1)$.
			We prove by induction on $k$ that for each $1 \leq k' \leq k \leq n$, $\Phi(k,k')\in \langle h_1,\dots,h_n\rangle$.
			When $k=1$, by definition
				$h_1=\Phi(1,1)$.
			Assume the statement is true for all $k<m\leq n$.
			By induction, $\Phi(m-1,m')\in [h_1,\dots,h_n]$ for all $1 \leq m'\leq m-1$.
			Note that $\Phi(m-1,m')x_{n-m'+1}$ is the sum of all monomials of degree $m$  in variables $x_1,\dots,x_{n-m'+1}$ divisible by $x_{n-m'+1}$.
			Hence, for each $1\leq m'\leq m-1$
			\begin{equation*}
				h_m-\Phi(m-1,1)x_n-\cdots-\Phi(m-1,m'-1)x_{n-m'+2}=\Phi(m,m').
			\end{equation*}
			At each stage of the proof the next $\Phi(k,k)$ is obtained as a sum of $h_k$ and polynomials obtained from $h_1,\dots,h_{k-1}$.
			Hence $\langle\Phi(1,1),\dots,\Phi(n,n)\rangle$ and $\langle h_1,\dots,h_n \rangle$ are equal.
		\end{proof}
			
		For integers $0\leq a\leq b$, denote by $h_a^b$ the complete homogeneous polynomial in variables $x_1,\dots,x_b$ of degree $a$.
		Then equation~(\ref{eq:OrigIdeal}) can be written as
		\begin{equation}\label{eq:SipleRedusingHomogenious}
			\langle h_1^n,\dots,h_n^n\rangle=\langle h_1^n,h_2^{n-1},\dots,h_n^1\rangle.
		\end{equation}
		A useful intermediate form of Theorem~\ref{thm:monomial sum} is separately set out in the next proposition.
			
		\begin{proposition}\label{prop:Homogeneous-1}
			For each $n\geq 1$,
			\begin{equation*}
				\langle h_1^n,\dots,h_n^n\rangle=\langle h_1^n,h_2^{n-1}\dots,h_n^{n-1}\rangle.
			\end{equation*}
		\end{proposition}
			
		\begin{proof}
			For each $1\leq i\leq n-1$
			\begin{equation*}
				h^n_{i+1}-x_n h^n_i=h^{n-1}_{i+1}.
			\end{equation*}
			We can rearrange the ideal to achieve the desired result by performing the above elimination in sequence on the ideal for $i={n-1}$ to $i=1$.
		\end{proof}
			
		\begin{remark}\label{remk:SymQotForms}
			By Theorem~\ref{thm:monomial sum} and Proposition~\ref{prop:Homogeneous-1} eliminating the last variable in $\mathbb{Z}[x_1,\dots,x_n]$,
			by rewriting $h_1$ as $x_n=-x_1-\cdots-x_{n-1}$ gives
			\begin{equation*}
				\frac{\mathbb{Z}[x_1,\dots,x_n]}{\langle h_1^n,\dots,h_n^n\rangle}
				\cong\frac{\mathbb{Z}[x_1,\dots,x_{n-1}]}{\langle h_2^{n-1},\dots,h_n^{n-1}\rangle}
				\cong\frac{\mathbb{Z}[x_1,\dots,x_{n-1}]}{\langle h_2^{n-1},\dots,h_n^{1}\rangle}.
			\end{equation*}
		\end{remark}

\section{Determining spectral sequence differentials}\label{sec:FreeLoopSU(n+1)/Tn}

    The aim of this section is to determine the differentials in a spectral sequence converging to the free loop cohomology $H^{*}(\Lambda(SU(n+1)/T^n);\mathbb{Z})$ for $n \geq 1$.
	The case when $n=0$ is trivial as $SU(1)$ is a point.
	The approach of the argument is similar to that of \cite{cohololgy_Lprojective},
	in which the cohomology of the free loop spaces of spheres and complex projective space are calculated using spectral sequence techniques. 
	However the details in the case of the complete flag manifold of the special unitary group are considerably more complex, requiring a sophisticated combinatorial argument arising form the structure of the complete homogeneous symmetric function not present for simpler spaces.

    \subsection{Differentials in the spectral sequence of the diagonal fibration}\label{sec:evalSS}

		For any space $X$, the map $eval \colon Map(I,X) \to X\times X$ is given by $\alpha \mapsto (\alpha(0),\alpha(1))$.
		It can be shown directly that $eval$ is a fibration with fiber $\Omega X$.
		In this section we compute the differentials in the cohomology Serre spectral sequence of this fibration for the case $X=SU(n+1)/T^n$.
		The aim is to compute $H^{*}(\Lambda(SU(n+1)/T^n);\mathbb{Z})$.
		The map $eval \colon \Lambda X \to X$ given by evaluation at the base point of a free loop is also a fibration with fiber $\Omega X$.
		This is studied in Section~\ref{sec:diff} by considering a map of fibrations from the evaluation fibration for $SU(n+1)/T^n$ to the diagonal fibration
		and hence the induced map on spectral sequences. For the rest of this section we consider the fibration
		\begin{equation}\label{eq:evalfib}
			\Omega (SU(n+1)/T^n) \to Map(I,SU(n+1)/T^n) \xrightarrow{eval} SU(n+1)/T^n\times SU(n+1)/T^n. 
		\end{equation}
		By extending the fibration $T^n \to SU(n+1) \to SU(n+1)/T^n$, we obtain the homotopy fibration sequence
		\begin{equation}\label{eq:SU/Tfib}
		    \Omega(SU(n+1)) \to \Omega(SU(n+1)/T^n) \to T^n \to SU(n+1).
		\end{equation}
		It is well known see \cite{CohomologyOmega(G/U)}, that the furthest right map above that is, the inclusion of the maximal torus into $SU(n+1)$ is null-homotopic.
		Hence there is a homotopy section $T^n \to \Omega(SU(n+1)/T^n)$.
		Therefore, as \eqref{eq:SU/Tfib}
		is a principle fibration, there is a space decomposition $\Omega(SU(n+1)/T^n) \simeq \Omega(SU(n+1)) \times T^n$.
		Using the K\"{u}nneth formula, we obtain the algebra isomorphisum
		\begin{equation}\label{eq:BaseLoopFlag}
			H^*(\Omega(SU(n+1)/T^n);\mathbb{Z}) \cong H^*(\Omega(SU(n+1);\mathbb{Z}) \otimes H^*(T^n;\mathbb{Z})
			\cong \Gamma_{\mathbb{Z}}[x'_2,x'_4,\dots,x'_{2n}] \otimes \Lambda_{\mathbb{Z}}(y'_1,\dots,y'_n)
		\end{equation}
		where $\Gamma_{\mathbb{Z}}[x'_2,x'_4,\dots,x'_{2n}]$ is the integral divided polynomial algebra on $x'_2,\dots,x'_{2n}$
		with $|x'_i|=i$ for each $i=2,\dots,2n$ and
		$\Lambda(y'_1,\dots,y'_n)$ is an exterior algebra generated by $y'_1,\dots,y'_n$
		with $|y'_j|=1$ for each $j=1,\dots,n$.
		Since
		\begin{equation*}
			Map(I,SU(n+1)/T^n)\simeq SU(n+1)/T^n
		\end{equation*}
		by Theorem~\ref{thm:H*SU/T} 
		all cohomology algebras of spaces in fibration (\ref{eq:evalfib}) are known.
		By studying the long exact sequence of homotopy groups associated to the fibration $T^n \to SU(n+1) \to SU(n+1)/T^n$,
		we obtain that $SU(n+1)/T^n$ and hence $SU(n+1)/T^n \times SU(n+1)/T^n$ are simply connected.
		Therefore the cohomology Serre spectral sequence of fibration (\ref{eq:evalfib}), denoted by $\{\bar{E}_r,\bar{d}^r\}$, 
		converges to $H^*(SU(n+1)/T^n;\mathbb{Z})$ with $\bar{E}_2$-page
		\begin{equation*}
		    \bar{E}^{p,q}_2=H^p(SU(n+1)/T^n\times SU(n+1)/T^n;H^q(\Omega(SU(n+1)/T^n);\mathbb{Z})).
		\end{equation*}
		In the following arguments we use the notation
			
		\begin{center}
			$H^*(Map(I,SU(n+1)/T^n);\mathbb{Z}) \cong \frac{\mathbb{Z}[\lambda_1,\dots,\lambda_{n+1}]}{\langle\sigma^{\lambda}_1,\dots,\sigma^{\lambda}_{n+1}\rangle}$
		\end{center}
		and
		\begin{center}
			$H^*(SU(n+1)/T^n \times SU(n+1)/T^n;\mathbb{Z}) \cong 
			\frac{\mathbb{Z}[\alpha_1,\dots,\alpha_{n+1}]}{\langle\sigma^{\alpha}_1,\dots,\sigma^{\alpha}_{n+1}\rangle} \otimes
			\frac{\mathbb{Z}[\beta_1,\dots,\beta_{n+1}]}{\langle\sigma^{\beta}_1,\dots,\sigma^{\beta}_{n+1}\rangle}
			$
		\end{center}
		where $|\alpha_i|=|\beta_i|=|\lambda_i|=2$ for each $i=1,\dots,n+1$ and $\sigma^{\lambda}_i,\sigma^{\alpha}_i$
		and $\sigma^{\beta}_i$ are the elementary symmetric polynomials in $\lambda_i,\alpha_i$ and $\beta_i$, respectively.
			
		\begin{center}
			\begin{tikzpicture}
				\matrix (m) [matrix of math nodes,
					nodes in empty cells,nodes={minimum width=5ex,
					minimum height=5ex,outer sep=-5pt},
					column sep=1ex,row sep=1ex]{
					&	 \vdots		&	 \vdots	 							&		 														 & 		  &						   &			 &					& \\
					&	  2n\;  	& \langle x'_{2n} \rangle&		 														 & 		  &						   &    	 &    			& \\
					&	 \vdots   &  \vdots	 							&		 														 & 		  &						   &   		 &  			  & \\
					&			6     &   \langle x'_6 \rangle &		 														 & 		  & 						 &\dots &   			  & \\
					H^{*}(\Omega(SU(n+1)/T^n;\mathbb{Z})) &			4     &   \langle x_4\rangle  &		 														 & 		  & 						 & 			 &					& \\
					&			2     &   \langle x'_2\rangle  &																 &\dots& 		 				 & 			 &   			  & \\
					&			      &  			   							&		 														 & 		  & 		 				 & 			 &          & \\
					&			1			&		\langle y'_i\rangle	&	\;\;\;			\lcdot	\;\;\;		 &\lcdot&\;\;\lcdot\;\;& \cdots& \lcdot		& \cdots \\	
					&		  0     &  			   							& \langle\alpha_i,\beta_i\rangle &\lcdot&\lcdot				 & \cdots&  \lcdot  & \cdots \\						
					&\quad\strut &     0    							&				2												 &		4	&	 6	 				 & \cdots&     2n   & \cdots \strut \\};
				\draw[-stealth] (m-2-3.south east) -- (m-8-8.north west);
				\draw[-stealth] (m-5-3.south east) -- (m-8-5.north);
				\draw[-stealth] (m-4-3.south east) -- (m-8-6.north);
				\draw[-stealth] (m-6-3.south) -- (m-8-4.north);
				\draw[-stealth] (m-8-3.south east) -- (m-9-4.north west);
				\draw[-stealth] (m-8-4.south east) -- (m-9-5.north west);
				\draw[-stealth] (m-8-5.south east) -- (m-9-6.north west);
				\draw[-stealth] (m-8-7.south east) -- (m-9-8.north west);
				\draw[thick] (m-1-2.east) -- (m-10-2.east) ;
				\draw[thick] (m-10-2.north) -- (m-10-9.north) ;
			\end{tikzpicture}
			\label{fig:evalSS}
		\end{center}
		\begin{center}
			$\;\;\;\;\;\;\;\;\;\;\;\;\;\;\;\;\;\;\;\;\;\;\;\;\;\;\;\;\;\;\;\;\;\;\;\;\;\;\;\;\; H^{*}(SU(n+1)/T^n\times SU(n+1)/T^n;\mathbb{Z})$
		\end{center}
		\begin{center}
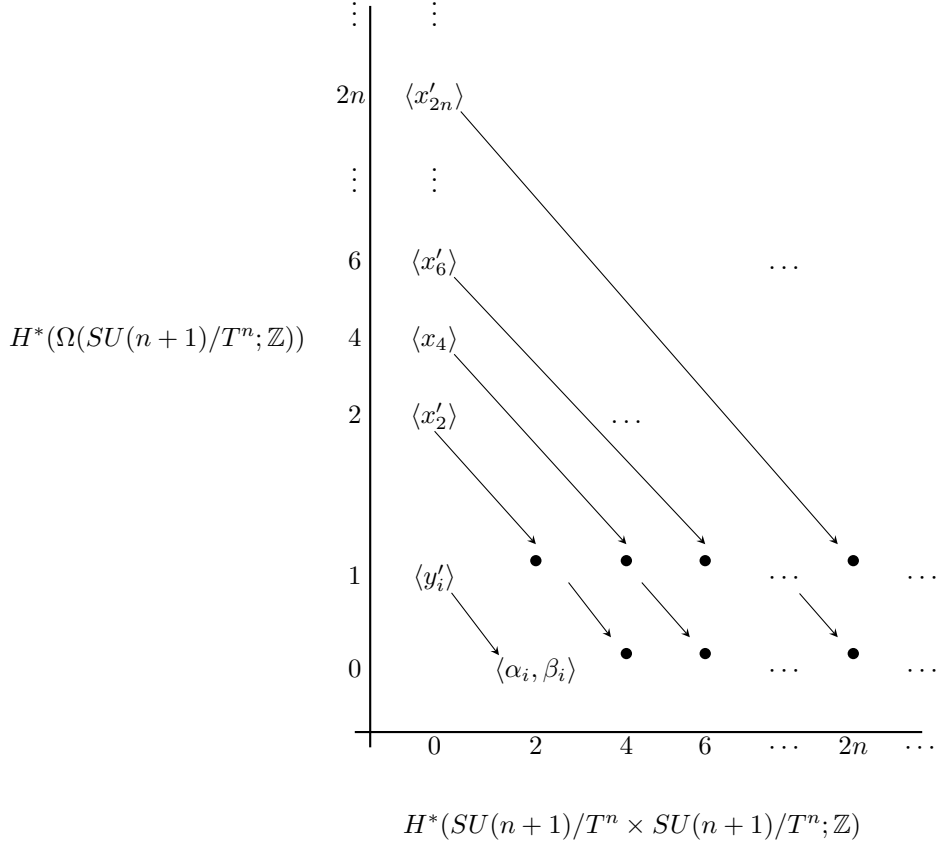

		    \captionof{figure}{Generators in the integral cohomology Leray-Serre spectral sequence $\{\bar{E}_r,\bar{d}^r\}$ converging to $H^{*}(Map(I,SU(n+1));\mathbb{Z})$.}
		\end{center}
			
		In the remainder of this section we describe explicitly the images of differentials shown in Figure~\ref{fig:evalSS}
		and show that all other differential not generated by these differentials using the Leibniz rule are zero.
		It will often be useful to use the alternative basis
		\begin{equation}\label{eq:ChangeBasis}
			v_i=\alpha_i-\beta_i \text{ and } u_i=\beta_i
		\end{equation}
		for $H^{*}(SU(n+1)/T^n\times SU(n+1)/T^n;\mathbb{Z})$, where $i=1,\dots,n+1$. 
		The following lemma determines completely the $\bar{d}^2$ differential on $\bar{E}_2^{*,1}$.	
			
		\begin{lemma}\label{lem:E^2_{*,1}d^2}
			With the notation above, in the cohomology Leray-Serre spectral sequence of fibration (\ref{eq:evalfib}),
			there is a choice of basis $y'_1,\dots,y'_n$ such that
		\begin{center}
			$\bar{d}^2(y'_i)=v_i$
		\end{center}
			for each $i=1,\dots,n$.
		\end{lemma}
				
		\begin{proof}
			There is a homotopy commutative diagram
	        \begin{equation*}\label{fig:evalcd}
			\xymatrix{
			{SU(n+1)/T^n} \ar[r]^(.37){\Delta}								   & {SU(n+1)/T^n\times SU(n+1)/T^n} \ar@{=}[d] \\
			{Map(I,SU(n+1)/T^n)} \ar[r]_(.42){eval} \ar[u]_{p_0} & {SU(n+1)/T^n\times SU(n+1)/T^n} }
		    \end{equation*}
		    where $p_0$, given by $\psi \mapsto \psi(0)$, is a homotopy equivalence and $\Delta$ is the diagonal map.
		    As the cup product is induced by the diagonal map, $eval^*$ has the same image as the cup product.
		    For dimensional reasons, $\bar{d}^2$ is the only possible non-zero differential ending at any $\bar{E}_*^{2,0}$ and no non-zero differential have domain in any $\bar{E}_*^{2,0}$.
		    Therefore in order for the spectral sequence to converge to $H^{*}(Map(I,SU(n+1)/T^n))$,
		    the image of $\bar{d}^2 \colon \bar{E}_2^{0,1} \to \bar{E}_2^{2,0}$ must be the kernel of the cup product on $H^*(SU(n+1)/T^n\times SU(n+1)/T^n;\mathbb{Z})$, 
		    which is generated by $v_1,\dots,v_n$.
		\end{proof}
				
		\begin{remark}\label{rmk:unique}
			The only remaining differentials on generators left to determine are those with domain in
			$\langle x'_2,x'_4\dots,x'_{2n} \rangle$ on some page $\bar{E}_r$ for $r \geq 2$.
			For dimensional reasons, the elements $x'_2,x'_4,\dots,x'_{2n}$ cannot be in the image of any differential.
			By Lemma~\ref{lem:E^2_{*,1}d^2}, the generators $u_1,\dots,u_n$ must survive to the $\bar{E}_{\infty}$-page,
			so generators $x'_2,x'_4,\dots,x'_{2n}$ cannot.	
			This is due to dimensional reasons combined with the fact that the spectral sequence must converge to $H^*(SU(n+1)/T^n)$.
			Now assume inductively that for each $i=1,\dots,n$ and $1\leq j<i$, $\bar{d}^{2j}$ is constructed.
			For dimensional reasons and due to all lower rows except $\bar{E}_r^{*,2}$ and $\bar{E}_r^{*,1}$ being annihilated by differentials already determined at lower values of
			$1\leq j<i$, the only possible non-zero differential beginning at $x'_{2i}$
			is $\bar{d}^{2i}:\bar{E}_{2i}^{0,2i}\to \bar{E}_{2i}^{2i,1}$.
			The image of each of the differentials $\bar{d}^{2i}$ will therefore be a unique class in $\bar{E}_{2i}^{2i,1}$
			in the kernel of $\bar{d}^2$ not already contained in the image of any $\bar{d}^r$ for $r<2i$.
			
			We note also that the elements $(x_{2i})_m$ for each $m \geq 2$ are also generators on the $\bar{E}_2$-page of the spectral sequence.
		    The differentials in the spectral sequence are also completely determined on all $(x_{2i})_m$ by their image on $x_{2i}$ in the following way.
			Using the relations $x_{2i}^m-m!(x_{2i})_m$ and the Leibniz rule, it follows that $\bar{d}^2(x_{2i}^m) = m\bar{d}^2(x_{2i})x_{2i}^{m-1}$ and hence again using the relations
			\begin{equation*}
                \bar{d}^2((x_{i2})_m) = \bar{d}^2(x_{2i})(x_{2i})_{m-1}.
            \end{equation*}
            Due to the fact that $\bar{d}^{2i}(x_{2i})$ must be non-trivial, there are no torsion elements on the spectral sequence pages $E_i^{*,*}$ and $(x_{2i})_m$ cannot be in the image of any differential,
            we obtain that
            $\bar{d}^r((x_{2i})_m) = \bar{d}^r(x_{2i})(x_{2i})_{m-1} = 0$ for $2\leq r<2i$.
            Therefore we can apply the same augments used to derive the $E_2$-page equation above on the $E_{2i}$-page we have that
            \begin{equation}\label{eq:DifOnDivPoly}
                \bar{d}^{2i}((x_{2i})_m) = \bar{d}^{2i}(x_{2i})(x_{2i})_{m-1}.
            \end{equation}
		\end{remark}
				
		We have $\bar{d}^2(u_i)=\bar{d}^2(v_i)=0$ and by Lemma~\ref{lem:E^2_{*,1}d^2} we may assume that $\bar{d}^2(y'_i)=v_i$ for each $i=1,\dots,n$.
		All non-zero generators $\gamma \in \bar{E}_2^{*,1}$ can be expressed in the form
		\begin{center}
			$\gamma = y'_k u_{i_1} \cdots u_{i_s} v_{j_1} \cdots v_{j_t}$
		\end{center}
		for some $1\leq k \leq n$, $1\leq i_1<\cdots<i_s\leq n$ and $1\leq j_1<\cdots<j_t\leq n$.
		Therefore, $\bar{d}^2(\gamma)$ is zero only if it is contained in
		$\langle \sigma^{\alpha}_1,\dots,\sigma^{\alpha}_{n+1}, \sigma^{\beta}_1,\dots,\sigma^{\beta}_{n+1}\rangle$.
		Hence it is important to understand the structure of the symmetric polynomials
		$\sigma^{\alpha}_1,\dots,\sigma^{\alpha}_{n+1}, \sigma^{\beta}_1,\dots,\sigma^{\beta}_{n+1}$.
		From Remark~\ref{remk:SymQotForms}, we see that
		polynomials $\sigma^{\alpha}_1$ and $\sigma^{\beta}_1$ can be thought of as expressions for $\alpha_{n+1}$ and $\beta_{n+1}$ in terms of the other generators of the ideal and generators 
		$\sigma^{\alpha}_2,\dots,\sigma^{\alpha}_{n+1}, \sigma^{\beta}_2,\dots,\sigma^{\beta}_{n+1}$
		can be replaced with homogeneous symmetric polynomials in $\alpha_1,\dots,\alpha_{n}$
		and $\beta_1,\dots,\beta_{n}$.
		It turns out that this rearranged basis is more convenient for deducing the differentials in the spectral sequence.
				
		Using the next two lemmas, we determine the correspondence between each $\sigma_l^\alpha$ and $\sigma_l^\beta$ and generators of the kernel of $\bar{d}^2$, which provides us with the image of the remaining differentials.
        For each $n \geq 1$, let $A=(a_1,\dots,a_n),B=(b_1,\dots,b_n)\in \mathbb{Z}^n_{\geq 0}$,
        such that not all of $a_1,\dots,a_n$ are zero.
        Denote by $L(a_1,\dots,a_n)$ the number of non-zero entries in $(a_1,\dots,a_n)$.
        Define the quotient of the permutation group on $n$ elements $\zeta_{(b_1,\dots,b_l)}^{(a_1,\dots,a_l)} \coloneqq S_n/\sim$ by
        \begin{equation*}
            \pi \sim \rho \iff (a_{\pi(1)},\dots,a_{\pi(n)})=(a_{\rho(1)},\dots,a_{\rho(n)})
            \text{ and } (b_{\pi(1)},\dots,b_{\pi(n)})=(b_{\rho(1)},\dots,b_{\rho(n)}).
        \end{equation*}
        Let $1\leq x(a_1,\dots,a_n) \leq n$ be the minimal integers such that
        \begin{equation}\label{eq:x}
            a_{x(a_1,\dots,a_n)} \geq 1.
        \end{equation}
        Notice that $x(a_{\pi(1)},\dots,a_{\pi(n)})=x(a_{\rho(1)},\dots,a_{\rho(n)})$ if $\pi \sim \rho$.
        Define element $s_{(b_1,\dots,b_n)}^{(a_1,\dots,a_n)}$ of $E_2^{2l-1,1}$ by
        \begin{equation*}
		    s_{(b_1,\dots,b_n)}^{(a_1,\dots,a_n)} \coloneqq
		    \sum_{\substack{\pi \in \zeta_{(b_1,\dots,b_n)}^{(a_1,\dots,a_n)}}}
		    {y'_{x(a_{\pi(1)},\dots,a_{\pi(n)})}v_{1}^{a_{\pi(1)}}\cdots v_{x(a_{\pi(1)},\dots,a_{\pi(n)})}^{a_{x(a_{\pi(1)},\dots,a_{\pi(n)})}-1}\cdots v_{n}^{a_{\pi(n)}} u_{1}^{b_{\pi(1)}}\cdots u_{n}^{b_{\pi(n)}}}.
	    \end{equation*}
			
	    \begin{lemma}\label{lem:SmallS}
	        With $s_{(b_1,\dots,b_n)}^{(a_1,\dots,a_n)}$ as given above,
	        \begin{align*}
	    	    \bar{d}^{2}&(s_{(b_1,\dots,b_n)}^{(a_1,\dots,a_n)})=
	    	    \\
		        &\sum_{\substack {\pi \in \zeta_{(b_1,\dots,b_l)}^{(a_1,\dots,a_l)} \\
		        0\leq t_j \leq a_j, \; 1\leq j \leq n}}
		        \Bigg{(}\prod_{1\leq k \leq n}{(-1)^{a_k-t_k}\binom{a_k}{t_k}}\Bigg{)}
		        \alpha_{1}^{t_{\pi(1)}}\cdots \alpha_{n}^{t_{\pi(n)}}
		        \beta_{1}^{b_{\pi(1)}+a_{\pi(1)}-t_{\pi(1)}}\cdots\beta_{n}^{b_{\pi(n)}+a_{\pi(n)}-t_{\pi(n)}}
		        .
	        \end{align*}
	    \end{lemma}
				
	    \begin{proof}
            Applying Lemma~\ref{lem:E^2_{*,1}d^2}, the Leibniz rule and the change of basis (\ref{eq:ChangeBasis}),
            \begin{align*}
               \bar{d}^{2}&(s_{(b_1,\dots,b_n)}^{(a_1,\dots,a_n)})= 
               \sum_{\substack{\pi \in \zeta_{(b_1,\dots,b_l)}^{(a_1,\dots,a_n)}}}
		       (\alpha_{1}-\beta_{1})^{a_{\pi(1)}} \cdots (\alpha_{n}-\beta_{n})^{a_{\pi(n)}}
		       \beta_{1}^{b_{\pi(1)}}\cdots \beta_{n}^{b_{\pi(n)}}.
            \end{align*}
            Using the binomial expansion on the terms $(\alpha_{i_j}-\beta_{i_j})^{a_j}$ for each $1\leq j \leq n$, we obtain
            \begin{align*}
	    	    \bar{d}^{2}(s_{(b_1,\dots,b_n)}^{(a_1,\dots,a_n)})= & \\
		        \sum_{\pi \in \zeta_{(b_1,\dots,b_n)}^{(a_1,\dots,a_n)}} &
		        \Bigg{(}
		        \bigg{(}\sum_{0\leq t\leq a_{\pi(1)}}{(-1)^{a_{\pi(1)}-t}\binom{a_{\pi(1)}}{t}\alpha_{1}^t\beta_{1}^{a_{\pi(1)}-t}}
		        \bigg{)}
		        \cdots \\ & \cdots
		        \bigg{(}\sum_{0\leq t\leq a_{\pi(n)}}{(-1)^{a_{\pi(n)}-t}\binom{a_{\pi(n)}}{t}\alpha_{n}^t\beta_{n}^{a_{\pi(n)}-t}}
		        \bigg{)}
		        \beta_{1}^{b_{\pi(1)}}\cdots\beta_{n}^{b_{\pi(n)}}
		        \Bigg{)}.
	        \end{align*}
	        The expression in the statement of the lemma now follows by collecting the terms
	        of the same product of $\alpha_{i_j}$'s and $\beta_{i_j}$'s.
        \end{proof}

       Using the notation above, define
        \begin{align}\label{eq:BigS}
            & S^{(c_1,\dots,c_l)} \coloneqq \nonumber \\
            & \sum_{\substack{ 0\leq t_j \leq c_j, \\ 1\leq j \leq n, \\ \text{some } t_j > 0 }}
            s_{(c_1-t_1,\dots,c_l-t_n)}^{(t_1,\dots,t_n)}
            \prod_{1 \leq i \leq n}
            \sum_{\substack{(a_1,\dots,a_{c_i}) \in \mathbb{Z}_{\geq 0} \\ |(a_1,\dots,a_{c_i})|=c_i-t_i, \\ a_{k-1} = 0 \implies a_{k} = 0, \\ 2 \leq k \leq c_i}}
            (-1)^{t_i+L(a_1,\dots,a_{c_i})}
            \binom{c_i}{a_1,\dots,a_{c_i}}.
        \end{align}
        The second differential when applied to $S^{(c_1,\dots,c_n)}$ gives an expression in terms
        of only products of $\alpha_{i_j}$'s or only $\beta_{i_j}$'s of type $(c_1,\dots,c_n)$.
        
        \begin{lemma}\label{lem:d2BigS}
            With $S^{(c_1,\dots,c_n)}$ given as above,
            \begin{equation*}
                \bar{d}^2(S^{(c_1,\dots,c_n)}) =
                \sum_{\pi \in \zeta_{(c_1,\dots,c_n)}^{(c_1,\dots,c_n)}}
                {\alpha_{1}^{c_{\pi(1)}}\cdots \alpha_{n}^{c_{\pi(n)}}}
                +(-1)^{c_1+\cdots+c_n+1}\sum_{\pi \in \zeta_{(b_1,\dots,b_n)}^{(a_1,\dots,a_n)}}
                {\beta_{1}^{c_{\pi(1)}}\cdots \beta_{n}^{c_{\pi(n)}}}.
            \end{equation*}
        \end{lemma}
    
        \begin{proof}
            It follows from Lemma~\ref{lem:SmallS} that the only monomials in $\bar{d}^2(S^{(c_1,\dots,c_l)})$ are integer multiples of
            \begin{equation}\label{eq:SumTerms}
               \alpha_{1}^{t_{\pi(1)}}\cdots \alpha_{n}^{t_{\pi(n)}}
               \beta_{1}^{c_{\pi(1)}-t_{\pi(1)}}\cdots \beta_{n}^{c_{\pi(n)}-t_{\pi(n)}} 
            \end{equation}
            for some $(t_1,\dots,t_n)\leq (c_1,\dots,c_n)$
            and $\pi \in \zeta_{(c_1-t_1,\dots,c_n-t_n)}^{(t_1,\dots,t_n)}$.
            When $(t_1,\dots,t_n)=(c_1,\dots,c_n)$, then the monomials in (\ref{eq:SumTerms}) are
            \begin{equation*}
                \alpha_{1}^{c_{\pi(1)}}\cdots \alpha_{n}^{c_{\pi(n)}}
            \end{equation*}
            for each $\pi \in \zeta_{(0,\dots,0)}^{(c_1,\dots,c_n)}$.
            By Lemma~\ref{lem:SmallS}, these monomial terms only appear in the image of the differential $\bar{d}^2$
            of $s^{(c_1,\dots,c_n)}_{(0,\dots,0)}$ and appear with multiplicity one.
            Hence the statement
            of the lemma
            holds when restricted to such monomials.
            
            By induction on $l=|(c_1,\dots,c_n)-(t_1,\dots,t_n)|$ we now prove that the statement holds for all other terms containing monomials in~(\ref{eq:SumTerms}).
            Take $0\leq (t'_1,\dots,t'_n) < (c_1,\dots,c_n)$
            such that $|(c_1,\dots,c_n)-(t'_1,\dots,t'_n)|=l\geq 1$.
            
            Since the terms of monomials in (\ref{eq:SumTerms}) for $(t'_1,\dots,t'_n)$
            only occur in the image of the $\bar{d}^2$ differential of $s^{(t_1,\dots,t_n)}_{(c_1-t_1,\dots,c_n-t_n)}$
            for $  0 < (t_1,\dots,t_n) \leq (c_1,\dots,c_n)$ such that $|(c_1,\dots,c_n)-(t_1,\dots,t_n)|\leq l$,
            the multiplicity of $s^{(t_1,\dots,t_n)}_{(c_1-t_1,\dots,c_n-t_n)}$ in (\ref{eq:BigS}) for
            $|(c_1,\dots,c_n)-(t_1,\dots,t_n)| > l$ does not effect the multiplicity of the terms
            corresponding to monomials with $(t'_1,\dots,t'_n)$ in (\ref{eq:SumTerms}).
            Hence, we need only show that the sum of terms of monomials in (\ref{eq:SumTerms}) for
            $(t'_1,\dots,t'_n)\neq(0,\dots,0)$ arising from
            $|(c_1,\dots,c_n)-(t'_1,\dots,t'_n)| > l$ is the negative of the constant
            \begin{equation}\label{eq:CancelCoef}
                K_{(t'_1,\dots,t'_n)} =
                \prod_{1 \leq i \leq n}
                \sum_{\substack{(a_1,\dots,a_{c_i}) \in \mathbb{Z}_{\geq 0} \\ |(a_1,\dots,a_{c_i})|=c_i-t'_i, \\ a_{k-1} = 0 \implies a_{k} = 0, \\ 2 \leq k \leq c_i}}
                (-1)^{t'_i+L(a_1,\dots,a_{c_i})}\binom{c_i}{a_1,\dots,a_{c_i}}.
            \end{equation}
           The only remaining case is the terms of monomials from (\ref{eq:SumTerms}) when $(t'_1,\dots,t'_n)=(0,\dots,0)$,
           which requires the sum to be $(-1)^{c_1+\cdots+c_n+1}$.
           Returning first to the main cases, the sum we are interested in is the same for any
           $\pi \in \zeta_{(c_1-t'_1,\dots,c_n-t'_n)}^{(t'_1,\dots,t'_n)}$.
           So without loss of generality, we assume that $\pi$ is the identity.
           In this case using the formula from Lemma~\ref{lem:SmallS}, we consider
           \begin{equation*}
                \sum_{\substack{(t'_1,\dots,t'_n) < (t_1,\dots,t_n) \leq (c_1,\dots,c_n)}}
                {K_{t_1,\dots,t_n}}\prod_{1\leq j \leq n}{(-1)^{t_j-t'_j}\binom{t_j}{t_j-t'_j}}
            \end{equation*}
            where we have additionally used the fact that $\binom{t_j}{t'_j}=\binom{t_j}{t_j-t'_j}$,
            which using (\ref{eq:CancelCoef}) can be inductively rewritten as
            \begin{align}\label{eq:InductionExpression}
                \sum_{(t'_1,\dots,t'_n) < (t_1,\dots,t_n) \leq (c_1,\dots,c_n)}
                \prod_{1\leq j \leq n}{(-1)^{t_j-t'_j}\binom{t_j}{t_j-t'_j}}
                \\ \cdot
                \prod_{1 \leq i \leq n}
                \sum_{\substack{(a_1,\dots,a_{c_i}) \in \mathbb{Z}_{\geq 0} \\ |(a_1,\dots,a_{c_i})|=c_i-t_i, \\ a_{k-1} = 0 \implies a_{k} = 0, \\ 2 \leq k \leq c_i}}
                (-1)^{t_i+L(a_1,\dots,a_{c_i})}
                \binom{c_i}{a_1,\dots,a_{c_i}}. \nonumber
            \end{align}
            Since $c_i-a_1-\cdots -a_{l_i}=t_i$, using equation~(\ref{eq:MultinomilExpansion}) we see that multiplying ${\binom{t_j}{t_j-t'_j}}$ by $\binom{c_i}{a_1,\cdots,a_{c_i}}$ in (\ref{eq:InductionExpression}) when $i=j$ is the same as extending each
            $a_1,\dots,a_{c_i}$ to $a_1,\dots,a_{c_i},t_j-t'_j$.
            Which means we can rewrite (\ref{eq:InductionExpression}) as
            \begin{align}\label{eq:InductionExpression2}
                \sum_{(t'_1,\dots,t'_n) < (t_1,\dots,t_n) \leq (c_1,\dots,c_n)}
                \prod_{1 \leq i \leq n}
                \sum_{\substack{(a_1,\dots,a_{c_i}) \in \mathbb{Z}_{\geq 0} \\ |(a_1,\dots,a_{c_i})|=c_i-t_i, \\ a_{k-1} = 0 \implies a_{k} = 0, \\ 2 \leq k \leq c_i}}
                (-1)^{t'_i+L(a_1,\dots,a_{c_i})}
                \binom{c_i}{a_1,\dots,a_{c_i}, t_i-t'_i}
            \end{align}
            where we also take the signs into account noting that
            \begin{equation*}
            (-1)^{t_i-t'_i}(-1)^{t'_i+L(a_1,\dots,a_{c_i})}
            =(-1)^{t'_i+L(a_1,\dots,a_{c_i})}.
            \end{equation*}
            Notice also that as $|(a_1,\dots,a_{l_i},t_i-t_i',0,\dots,0)|=c_i-t'_i$, once $t_i-t'_i$ is moved to the left of $a_k = 0$, the multinomial type coefficient terms appearing in (\ref{eq:InductionExpression2}) are the same as those appearing in
            (\ref{eq:CancelCoef}).
            It remains to check that the multiplicity of the sum of these coefficients in (\ref{eq:InductionExpression2}) agrees with (\ref{eq:CancelCoef}).
            
            Each multinomial type coefficient in (\ref{eq:CancelCoef}) is determined by a choice of sequence $(a_1,\dots,a_{c_i})$ for $1\leq i \leq n$.
            Once the product is expanded, the multinomial type coefficient terms in (\ref{eq:InductionExpression}) are the product of $n$ coefficients corresponding to $n$ sequences
            \begin{equation*}
                (a_1,\dots,a_{c_i})=(a_1,\dots,a_{l_i},0,\dots,0),
            \end{equation*}
            for each $1\leq i \leq n$ and some $1\leq l_i\leq c_i$.
            The coefficient product terms in (\ref{eq:InductionExpression2}) of the same form correspond to $n$ sequences of the form
            \begin{equation}\label{eq:SubSequences}
               (a_1,\dots,a_{l_i},0,\dots,0) \text{ or } (a_1,\dots,a_{l_i-1},0,\dots,0)
            \end{equation}
            for $1 \leq i \leq n$, the second case corresponding to extending by $t_i-t'_i>0$ and the first to extending by $t_i-t'_i=0$.
            Note that for at least some $i$ the second case must be chosen as $(t'_1,\dots,t'_n)<(t_1,\dots,t_n)$.
            However, as we fix the order of the $n$ sequences all other possible combinations of choices occur exactly once.
            
            The sign of each of the terms from (\ref{eq:CancelCoef}) inside the product is the product of $(-1)^{t'_i+L(a_1,\dots,a_{c_i})}=t'_i+l_i$ changes from those in (\ref{eq:InductionExpression2}) 
            each time the first choice in (\ref{eq:SubSequences}) is taken as for non-zero $t_j-t'_j$, $L(a_1,\dots,a_{l_i}-1,t_j-t'_j,0,\dots,0) = L(a_1,\dots,a_{l_i}-1,0,\dots,0)+1$.
            The number of ways to obtain a unique product of coefficients from (\ref{eq:CancelCoef}) in (\ref{eq:InductionExpression2}) is exactly all choices possible in (\ref{eq:SubSequences}).
            This choice corresponds to picking subsets of an $n$ set that is not the entire set.
            As this sign depends on the size of the chosen subset, the total sum of these terms is the alternating sum of the $n^{\text{th}}$ binomial coefficients (\ref{eq:binom}),
            without the first term in the sum.
            The result is therefore a single term of multiplicity $1$ up to sign.
            This sign is the opposite of that of the first missing term from the alternating sum of binomial coefficients, which corresponds to the term in (\ref{eq:CancelCoef}) as required.
            
           Finally it remains to consider the case when $(t'_1,\dots,t'_n)=(0,\dots,0)$.
           Using the result of the previous part of the proof, substituting $t'_i=0$ in (\ref{eq:CancelCoef}), this will be the negative of 
           \begin{equation*}
                \prod_{1 \leq i \leq n}
                \sum_{\substack{(a_1,\dots,a_{c_i}) \in \mathbb{Z}_{\geq 0} \\ |(a_1,\dots,a_{c_i})|=c_i, \\ a_{k-1} = 0 \implies a_{k} = 0, \\ 2 \leq k \leq c_i}}
                (-1)^{L(a_1,\dots,a_{c_i})}
                \binom{c_i}{a_1,\dots,a_{c_i}}
                =(*).
           \end{equation*}
           Since now $a_1+\dots+a_{c_i}=c_i$, the coefficients $\binom{c_i}{a_1,\dots,a_{c_i}}$ are genuine multinomial coefficients.
           Hence, by collecting all sequence with the constant $L(a_1,\dots,a_{c_i})=j$ using equation~(\ref{eq:StirlingExpansion}) then Lemma~\ref{lem:StirlingIdentity} to rewrite
           the equation above as
           \begin{equation*}
               (*) =
               \prod_{1\leq i \leq n}\sum_{1\leq j \leq c_i}(-1)^j j!\stirling{c_i}{j}
               =
               \prod_{1\leq i \leq n}(-1)^{c_i} = (-1)^{c_1+\cdots+c_n}
           \end{equation*}
           which completes the proof.
        \end{proof}
        
        We can now describe explicitly the image of $\bar{d}^2$ in the spectral sequence.
         
        \begin{theorem}\label{thm:d^2Image}
            The only non-zero differentials on the generators $x'_{2l}$ for $2\leq l\leq n+1$ are given by
            \begin{equation*}
                \bar{d}^{2(l-1)}(x'_{2(l-1)})
                =\sum_{\substack{|(c_1,\dots,c_{l'})|=l}}{S^{(c_1,\dots,c_l)}}.
            \end{equation*}
        \end{theorem}
        
        \begin{proof}
            Applying Lemma~\ref{lem:d2BigS}, the image under $\bar{d}^2$ of the right hand side is the difference of complete homogeneous symmetric functions in $l$ variables with $\alpha$'s and $\beta$'s respectively.
            By Remark~\ref{rmk:unique}, since no expression in only $\alpha$'s or $\beta$'s lies in the image of $\bar{d}^2$, the image of the difference of the homogeneous symmetric polynomials in $\alpha$'s and $\beta$'s must generate the image of $x'_{2l}$ under $\bar{d}^{2(l-1)}$.
        \end{proof}
        
    \subsection{Differentials for the free loop spectral sequence}\label{sec:diff}
			
			Next we consider the map $\phi$ of the evaluation fibration of $SU(n+1)/T^n$ for $n \geq 1$ 
			and the diagonal fibration studied in Section~\ref{sec:evalSS}
			given by the following commutative diagram
			\begin{equation*}\label{fig:fibcdSU}
				\xymatrix{
					{\Omega(SU(n+1)/T^n)} \ar[r]^(.5){} \ar[d]^(.45){id} & {\Lambda(SU(n+1)/T^n)} \ar[r]^{eval} \ar[d]^(.45){exp}  & {SU(n+1)/T^n} \ar[d]^(.45){\Delta} \\
					{\Omega(SU(n+1)/T^n)} \ar[r]^(.45){}   							 & {Map(I,SU(n+1)/T^n)} 	 \ar[r]^(.425){eval}  					 & {SU(n+1)/T^n\times SU(n+1)/T^n}}
			\end{equation*}
			where $\exp$ is given by $\exp(\alpha)(t)=\alpha(e^{2\pi i t})$.
			As $SU(n+1)/T^{n}$ is simply connected, the evaluation fibration induces a cohomology Leray-Serre spectral sequence $\{ E_r,d^r \}$.
			Hence $\phi$ induces a map of spectral sequences $\phi^*:\{ \bar{E}_r,\bar{d}^r \} \to \{ E_r,d^r \}$. 
			More precisely for each $r\geq 2$ and $a,b \in \mathbb{Z}$, there is a commutative diagram
			\begin{equation}\label{fig:phicd}
				\xymatrix{
				    {\bar{E}_r^{a,b}} \ar[r]^(.4){d^r} \ar[d]^(.46){\phi^*} & {\bar{E}_r^{a+r,b-r+1}} \ar[d]^(.46){\phi^*}  \\
					{E_r^{a,b}} \ar[r]^(.4){\bar{d}^r} & {E_r^{a+r,b-r+1}}}
			\end{equation}
			where $\phi^*$ for each successive $r$ is the induced map on the homology of the previous page, beginning as the map induced on the tensor on the $\bar{E}_2$-pages
			by the maps $id \colon \Omega(SU(n+1)/T^n)\to\Omega(SU(n+1)/T^n)$ and $\Delta \colon SU(n+1)/T^n \to SU(n+1)/T^n \times SU(n+1)/T^n$.
			For the remainder of the section we use the notation
			\begin{center}
			$H^*(\Omega(SU(n+1)/T^n);\mathbb{Z})\cong \Gamma_{\mathbb{Z}}(x_2,x_4,\dots,x_{2n})\otimes\Lambda_{\mathbb{Z}}(y_1,\dots,y_{n})$
			\end{center}
			\begin{center}
			and $\;\;\; H^*(SU(n+1)/T^n;\mathbb{Z})\cong \frac{\mathbb{Z}[\gamma_1,\dots,\gamma_{n+1}]}{\langle\sigma^\gamma_1,\dots,\sigma_{n+1}^\gamma\rangle}$,
			\end{center}
			where $|y_i|=1, |\gamma_j|=2, |x_{2i}|=2i$ for each $1\leq i\leq n,1\leq j\leq n+1$ and $\sigma^\gamma_1,\dots,\sigma_{n+1}^\gamma$
			are the basis of elementary symmetric functions on $\gamma_i$.
			We now determine all the differentials in $\{ E_r,d^r \}$.	
			
			\begin{theorem}\label{thm:allDiff}
			For each integer $n\geq 1$ and for $2 \leq l \leq n+1$, the only non-zero differentials on generators of the $E_2$-page of $\{ E_r,d^r \}$ are up to class representative and sign are given by
				\begin{equation*}
					d^{2(l-1)}(x_{2(l-1)}) =
		        	\sum_{\substack{|(c_1,\dots,c_{n})|=l, \\ 1\leq j \leq n, \; c_j \geq 1}}
					{c_jy_{j}\gamma_1^{c_1}\cdots\gamma_j^{c_j-1}\cdots\gamma_n^{c_{n}}}.
				\end{equation*}
			\end{theorem}
            
            \begin{proof}
                The identity $id \colon \Omega(SU(n+1)/T^n)\to\Omega(SU(n+1)/T^n)$ induces the identity map on cohomology.
				The diagonal map $\Delta \colon SU(n+1)/T^n \to SU(n+1)/T^n \times SU(n+1)/T^n$ induces the cup product on cohomology.
				Hence by choosing appropriate generators of $\{ E_r,d^r \}$, we assume that
				\begin{center}
				    $\phi^*(y'_i)=y_i,\;\;\phi^*(x'_i)=x_i\;\;$and$\;\;\phi^*(\alpha_i)=\phi^*(\beta_i)=\phi^*(u_i)=\gamma_i,\;\;$so$\;\;\phi^*(v_i)=0$.
				\end{center}
				For dimensional reasons, the only possibly non-zero differential on generators ${y}_i$ in $\{ E_r,d^r \}$ is $d^{2}$.
				However for each $1\leq i\leq n$ using commutative diagram (\ref{fig:phicd}) and Lemma~\ref{lem:E^2_{*,1}d^2}, we have
				\begin{center}
					$d^2(y_i)=d^2(\phi^*(y'_i))=\phi^*(\bar{d}^2(y'_i))=\phi^*(v_i)=0$.
				\end{center}
				Additionally all differentials on generators $\gamma_i$, for each $1\leq i \leq n+1$, are zero for dimensional reasons.
				Hence all elements of $E_2^{*,1}$ and $E_2^{*,0}$ survive to the $E_{\infty}$-page unless they are in the image of some differential
				$d^r$ for $r\geq 2$.
				Using commutative diagram (\ref{fig:phicd}), we have up to class representative and sign
				\begin{equation*}
					d^{2(l-1)}(x_{2(l-1)})=\phi^*(\bar{d}^{2(l-1)}(x'_{2(l-1)})).
				\end{equation*}
				Since $\phi^*(v_i)=0$, an expression for $\phi^*(\bar{d}^{2(l-1)}(x'_{2(l-1)}))$ is obtained form the terms in the expression of Theorem~\ref{thm:d^2Image} that do not contain $v_i$.
                Therefore
                \begin{align}\label{eq:ImediateDiff}
					\nonumber
					&d^{2(l-1)}(x_{2(l-1)}) = \\
		        	&\sum_{\substack{|(c_1,\dots,c_{n})|=l, \\ 1\leq j \leq n, \; c_j \geq 1}}
					{y_{j}\gamma_1^{c_1}\cdots\gamma_j^{c_j-1}\cdots\gamma_n^{c_{n}}}
					\prod_{1 \leq i \leq n}
                    \sum_{\substack{(a_1,\dots,a_{c_i}) \in \mathbb{Z}_{\geq 0} \\ |(a_1,\dots,a_{c_i})|=c_i-\mathbb{I}_{i=j}, \\ a_{k-1} = 0 \implies a_{k} = 0, \\ 2 \leq k \leq c_i}}
                    (-1)^{\mathbb{I}_{i=j}+L(a_1,\dots,a_{c_i})}
                    \binom{c_i}{a_1,\dots,a_{c_i}}
				\end{align}
				where $\mathbb{I}_{i=j}$ is the indicator function equal if $1$ if $i=j$ and $0$ otherwise.
				Notice that by applying (\ref{eq:StirlingExpansion}) to $(a_1,\dots,a_{c_i})$ with the same value of $L(a_1,\dots,a_{c_i})$ and then using Lemma~\ref{lem:StirlingIdentity}, when $i \neq j$, we have 
				\begin{equation*}
			        \sum_{\substack{(a_1,\dots,a_{c_i}) \in \mathbb{Z}_{\geq 0} \\ |(a_1,\dots,a_{c_i})|=c_i \\ a_{k-1} = 0 \implies a_{k} = 0, \\ 2 \leq k \leq c_i}}
                    (-1)^{L(a_1,\dots,a_{c_i})}
                    \binom{c_i}{a_1,\dots,a_{c_i}}
                    =
                    \sum_{1\leq L \leq c_i}(-1)^L\stirling{c_i}{L}
                    =
                    (-1)^{c_i}
                    .
				\end{equation*}
				Additionally, when $i=j$ and using (\ref{eq:MultinomilExpansion}) we have
				\begin{align*}
				    &\sum_{\substack{(a_1,\dots,a_{c_i}) \in \mathbb{Z}_{\geq 0} \\ |(a_1,\dots,a_{c_i})|=c_i-1 \\ a_{k-1} = 0 \implies a_{k} = 0, \\ 2 \leq k \leq c_i}}
                    (-1)^{1+L(a_1,\dots,a_{c_i})}
                    \binom{c_i}{a_1,\dots,a_{c_i}}
                    \\
                    =&
                    \sum_{\substack{(a_1,\dots,a_{c_i}) \in \mathbb{Z}_{\geq 0} \\ |(a_1,\dots,a_{c_i})|=c_i-1 \\ a_{k-1} = 0 \implies a_{k} = 0, \\ 2 \leq k \leq c_i}}
                    (-1)^{1+L(a_1,\dots,a_{c_i})}
                    \binom{c_i}{1}
                    \binom{c_i-1}{a_1,\dots,a_{c_i}}
                    \\ =&
                    c_i\sum_{1\leq L \leq c_i-1}(-1)^{L+1} \stirling{c_i-1}{L}
                    =
                    c_i(-1)^{c_i}
				\end{align*}
				which together reduce (\ref{eq:ImediateDiff}) to 
				\begin{equation*}
					d^{2(l-1)}(x_{2(l-1)}) =
		        	\sum_{\substack{|(c_1,\dots,c_{n})|=l, \\ 1\leq j \leq n, \; c_j \geq 1}}
					{(-1)^{\lceil \frac{n}{2} \rceil}c_jy_{j}\gamma_1^{c_1}\cdots\gamma_j^{c_j-1}\cdots\gamma_n^{c_{n}}}.
				\end{equation*}
				Since all terms have the same sign we choose the positive generator, obtaining the formula in the statement of the theorem.
            \end{proof}
        
        \begin{remark}\label{rmk:DifOnDivPoly}
            By applying $\phi^*$ to equation \eqref{eq:DifOnDivPoly} we obtain additionally for each $m\geq 2$ that
            \begin{equation*}
                d^{2(l-1)}((x_{2(l-1)}))
                =
                d^{2(l-1)}(x_{2(l-1)})(x_{2(l-1)})_{m-1}.
            \end{equation*}
        \end{remark}
        
        \section{Spectral sequence computations with Gr{\"o}bner basis}\label{sec:SpectralGrobner}
		
		The following proposition describes the application of Gr\"{o}bner basis to spectral sequences making use of the theory in Section~\ref{sec:Grobner}, motivates Theorem~\ref{thm:Ideals} and will be heavy used in Section~\ref{sec:LSU4/T3} of the paper.
		A library of code for preforming computations outlined in the proposition can be found at \cite{Burfitt2021}.
	    The proposition applies straightforwardly to the Leray-Serre spectral sequence and is stated here for cohomology spectral sequences of algebras though could be reformulated through duality for a homology spectral sequence of algebras, swapping the base space with the fibre and rows for columns. 

	    \begin{proposition}\label{thm:SpectralGrobner}
	        Let $\{ E_r, d^r \}$ be a first quadrant spectral sequence of algebras in coefficients $R$
	        for which the $E^2$-page is of the form $E_2^{p,q} = E_2^{p,0}\otimes E_2^{0,q}$.
	        Suppose also that $E_2^{*,0}$ can be expressed as the quotient of a polynomial algebra
	        \begin{equation*}
	            A = \frac{R[x_1,x_2,\dots,x_n]}{I}
	        \end{equation*}
	        for some ideal $I$ in $R[x_1,x_2,\dots,x_n]$.
	        For some $a\geq 0$,
	        let $S_a$ be the additive generators of $E_r^{0,a}\setminus\{x_1,x_2,\dots,x_n\}$.
	        The computation within each row may be broken down as follows.
	        The kernel of $d^r$ in $A$ on row $E_{r}^{*,a}$ is generated by
	        \begin{enumerate}
    	        \item
    	            the kernel of the map 
    	            \begin{equation*}
    	                \phi_a^r \colon R[S_{a},x_1,x_2,\dots,x_n] \to R[S_{a-r+1},x_1,x_2,\dots,x_n]
    	            \end{equation*}
    	            induced by $d^r$ and
    	        \item
    	            the pre-image under $d^r$ of generators of the intersection
    	            \begin{equation*}
    	                \im(\phi_a^r) \cap S_{a-r+1}I
    	            \end{equation*}
    	            where $S_{a}I= \{ si \; | \; s\in S_{a-r+1}, i \in I  \}$.
    	            Treating $S_{a-r+1}I$ as an ideal generated by elements $si$ such that $s\in S_{a-r+1}$ and $i$ is a generator of $I$, the intersection can be computed by Gr\"{o}bner basis as an intersection of ideals
    	            noting that any element that increases the degree of the polynomial element of $S_{a-r+1}$ is no longer on row $a$, hence is ignored as it is justified below.
	        \end{enumerate}
	        The algebra structure on the $E_{r+1}$-page is generated by
	        the union of row generators given above and $S_{a}I \cap E_r^{a,*}$ for each $a\geq 0$.
            The relations for the algebra structure on $E_{r+1}$-page are given by the union
            of $I,\: \im(d^b)$ for $2\leq b \leq r$ and relations of the first column $E_{2}^{*,0}$ intersected with the generators of the $E_{r+1}$-page algebra.

	        In addition when $R=\mathbb{Z}$, the type of torsion present in each row $E_r^{a,*}$ and their generating relations can be obtained from the reduced Gr\"{o}bner basis of the ideal
	        \begin{equation}\label{eq:GrobnerTorsionPart}
	            \langle I\cap E_r^{a,*},\: \im(d^2)\cap E_r^{a,*},\: \im(d^3)\cap E_r^{a,*}, \dots,\: \im(d^r)\cap E_r^{a,*} \rangle
	        \end{equation}
	        as the non-unit greatest common divisor of the absolute values of the coefficients of each element of the reduced Gr\"{o}bner basis.
	        In this case the coefficients greatest common divisors describe all possible torsion types present and their multiplicities.
	    \end{proposition}
	
	    For part (1) of the proposition, the kernel of a morphism of polynomial rings can also be obtained using Gr\"{o}bner basis, see for example \cite{Biase05}.
	    However such computations will be unnecessary in the present work.
	    During a Gr\"{o}bner basis computation of the intersection of ideals in part (2), expression containing elements in $E_r^{*,b}$ for $b>a$ may be considered, these can be discarded immediately, since reduction, $S$-polynomials and $G$-polynomials of such elements can never decease the row $b$ during the procedure without reducing the entire polynomial to zero.
	    Doing this greatly speeds up the execution of the algorithm.

	    \begin{proof}
	        For each $a\geq 0$ and $r\geq 2$,
	        the following diagram commutes
				\begin{equation*}
				    \xymatrix{
						{\ker(\phi_a^r)} \ar[r]^(.5){} \ar[d]^(.45){q} & {R[S_{a},x_1,x_2,\dots,x_n]} \ar[r]^(0.46){\phi_a^r} \ar[d]^(.45){q}  & {R[S_{a-r+1},x_1,x_2,\dots,x_n]} \ar[d]^(.45){q} \\
						{\ker(d^r)} \ar[r]^(.45){} & {E_r^{a,*}} \ar[r]^-(.5){d^r} & {E_r^{a-r+1,*}}}
				\end{equation*}
			where $q$ is the quotient map by $I$.
			Elements of $\ker(\phi_a^r)$ remain in $\ker(d^r)$ after quotient map $q$. Other element of $R[S_{a-r+1},x_1,x_2,\dots,x_n]$ not in $\ker(\phi_a^r)$ that are in the kernel of $d^r$ under $q$ have non-trivial image under $\phi^r_n$ and so are contained in $S_{a-r+1}I$.
			Hence parts (1) and (2) together describe the whole kernel of $d^r$ on row $E^{a,*}_{r}$.
			The relations form $I$ and $\im(d^b)$ span all relations in the spectral sequence by construction and additional generators $S_{a-r+1}I \cap E_r^{a,*}$ are formally added to ensure that the all relations are contained in the span of the generators.
			
			Finally, the torsion in the case $R=\mathbb{Z}$ can be seen to be correct by the following inductive argument to $E$-reduce any minimal generating set $T_1,\dots,T_m$ of the torsion relations to the torsion polynomials in the minimal Gr\"{o}bner basis $G$. 
			The leading term of $T_1$ must be divisible by the leading term of a unique element $g$ in $G$.
			As $T_1$ can be reduced to $0$ by $G$, reducing its leading term to be equal to that of $g$ and then fully $E$-reducing it by $G\setminus g$ must leave exactly $g$.
			We follow the same steps to reduce the rest of $T_2,\dots,T_m$ to element of $G$, except we first $E$-reduce each $T_i$ fully by all elements of $g$ already obtained form $T_1,\dots,T_{i-1}$ before beginning. 
			No element of $T_i$ can be reduced to zero by previously obtained element of $g$ as we assume that $T_1,\dots,T_m$ is minimal and there can be no additionally torsion element of $G$ not obtained form reduction of $T_1, \dots, T_m$ as it is assumed to generate the torsion relations.
	    \end{proof}

        Though not necessary in the course of this work, in some cases the procedure outlined in Proposition~\ref{thm:SpectralGrobner} can require an additional computational setup which can again be done algorithmicly as described in the flowing remark.

        \begin{remark}\label{rmk:Pre-ImageReduction}
            Note that the pre-image of the Gr\"{o}bner basis of $\im(\phi_a^r) \cap S_{a-r+1}I$ in part (2) of Proposition~\ref{thm:SpectralGrobner} can be uniquely obtained computationally by fully reducing it by a Gr\"{o}bner basis of $\im(\phi_a^r)$ to $0$ while keeping track of the reductions.
            The generators of $\im(\phi_a^{r})$ need not be a Gr\"{o}bner basis and since the reduction of $\im(\phi_a^r) \cap S_{a-r+1}I$ is in terms of a Gr\"{o}bner basis of $\im(\phi_a^r)$,
            the resulting expression from the reduction is in terms of their Gr\"{o}bner basis not the original generators.
            For the computations made in the paper we will only need to consider the case where the generators of $\im(\phi_a^{r})$ form a Gr\"{o}bner basis.
            
            However an expression in term of the original generators can be deduced as follows.
            Track the computation of the reduced Gr\"{o}bner basis of $\im(\phi_a^{r})$ to obtain an expression for the Gr\"{o}bner basis in terms of the original generators.
            This expression is unique up to Syzygys of the Gr\"{o}bner basis which can be computed through a reduction procedure, see for example \cite{Erocal16},
            the algorithm being equally valid over $\mathbb{Z}$.
        \end{remark}
        
        Though Proposition~\ref{thm:SpectralGrobner} provides a procedure for computations with spectral sequences of the appropriate form, the later calculations in this work are obtained under stricter assumptions that greatly enhance the computation efficacy of integral Gr\"{o}bner basis as detailed in the next remark.
        
        \begin{remark}\label{rmk:HomGrobSimplification}
            When the generators $x_1,x_2\dots,x_n$ of the algebra $A$ in Proposition~\ref{thm:SpectralGrobner} have the same degree, the representatives of generators of $\im(\phi_a^r)$ are homogeneous in their $x_1,x_2\dots,x_n$ components.
            Reduction preserves homogeneity and
            in addition, $S$-polynomial and $G$-polynomials of homogeneous polynomials used during the Gr\"{o}bner basis computation are again homogeneous polynomials and only ever increase the degree.
            Assuming also that generators of $I$ are homogeneous and that elements of $A$ have bounded degree,
            then when computing the Gr\"{o}bner basis with such polynomials to obtain $\im(\phi_a^r) \cap S_{a-r+1}I$ in part (2) of the proposition, we may discard any $S$-polynomial or $G$-polynomials of homogeneous degree greater than the maximal degree in $A$, greatly speeding up the computation.
        \end{remark}
        
        \section{Basis}\label{sec:basis}
			
			We now develop theory to reveal the structure of the cohomology of the free loop space of $SU(n+1)/T^n$, which culminates in Theorem~\ref{thm:Ideals}.
			To begin we consider a basis of $\mathbb{Z}[\gamma_1,\dots,\gamma_n]$ that resembles the image of the $d^2$ differential in Theorem~\ref{thm:allDiff},
			as it becomes easier to study the $E_3$-page and subsequent pages of the spectral sequence, which in addition to making possible theoretical results also simplifies all computations in Section~\ref{sec:LSU4/T3}.
			
			\begin{remark}\label{rmk:TildeBasis}
				In $\mathbb{Z}[\gamma_1,\dots,\gamma_{n}]$, let $\bar{\gamma}=\gamma_1+\cdots+\gamma_n$ and $\tilde{\gamma}_i=\bar{\gamma}+\gamma_i$ for each $1\leq i\leq n$. 
				We rearrange the standard basis $\gamma_1,\dots,\gamma_n$ of $\mathbb{Z}[\gamma_1,\dots,\gamma_{n}]$ to $\gamma_1,\dots,\gamma_{n-1},\bar{\gamma}$
				and then rearrange to $\tilde{\gamma}_1,\dots,\tilde{\gamma}_{n-1},\bar{\gamma}$,
				by adding $\bar{\gamma}$ to all other basis elements.
				Notice that the replacement $\gamma_i \mapsto \tilde{\gamma}_i$ for $1\leq i \leq n-1$, $\gamma_n \mapsto \bar{\gamma}$
				could have instead been chosen as $\gamma_j \mapsto \bar{\gamma}$ for any $1\leq j \leq n$ and $\gamma_i \mapsto \tilde{\gamma}_i$ for any $i\neq j$.
				Furthermore, replacing $\bar{\gamma}$ by $\tilde{\gamma}_n = (n+1)\bar{\gamma}-\tilde{\gamma}_1-\cdots-\tilde{\gamma}_{n-1}$ gives $\tilde{\gamma}_n$,
				so $\tilde{\gamma}_1,\dots,\tilde{\gamma}_{n}$ forms a rational basis.
			\end{remark}
			
			\begin{proposition}\label{prop:basis}
				Using the notation of equation~(\ref{eq:SipleRedusingHomogenious}),
				we can rewrite $h^{n-l+2}_i$ for each $2\leq l \leq n+1$ in the basis of Remark~\ref{rmk:TildeBasis}, as
				\begin{equation*}
					h^{n-l+2}_l=\sum_{\substack{0\leq k \leq l \\ 
					1\leq i_1 \leq \cdots \leq i_k \leq n-l+1}}
					{(-1)^{l-k} \binom{n+1}{l-k} \tilde{\gamma}_{i_1}\cdots\tilde{\gamma}_{i_{k}}\bar{\gamma}^{l-k}}.
				\end{equation*}
			\end{proposition}
			\begin{proof}
				First note that in the basis of Remark~\ref{rmk:TildeBasis}, we can rewrite the original basis in terms of the new one by
				\begin{equation}\label{eq:BaseExpression}
					\gamma_i=\tilde{\gamma_i}-\bar{\gamma} \text{  for  } 1\leq i\leq n-1, \;\;\; \gamma_n=n\bar{\gamma}-\sum_{i=1}^{n-1}{\tilde{\gamma_i}}.
				\end{equation}
				When $l=2$ using rearrangement (\ref{eq:BaseExpression}),
				\begin{flalign}
					h_2^{n}&=\sum_{a=0}^{2}{\big{(}(n\bar{\gamma}-\sum_{j=1}^{n-1}{\tilde{\gamma}_j})^{2-a}
					\sum_{1\leq i_1\leq i_2 \leq n-1}{\prod^a_{k=1}{(\tilde{\gamma}_{i_k}-\bar{\gamma})}}\big{)}} \nonumber \\
					&= (n\bar{\gamma}-\sum_{j=1}^{n-1}{\tilde{\gamma}_j})^{2}
						+\sum^{n-1}_{a=1}{(n\bar{\gamma}-\sum_{j=1}^{n-1}{\tilde{\gamma}_j})(\tilde{\gamma}_a-\bar{\gamma})}
						+\sum^{n-1}_{a=1}{(\tilde{\gamma}_{a}-\bar{\gamma})^2}
						+\sum_{1\leq i_1 < i_2 \leq n-1}{(\tilde{\gamma}_{i_1}-\bar{\gamma})(\tilde{\gamma}_{i_2}-\bar{\gamma})}. \label{eq:l=2}
				\end{flalign}
				For $1\leq k,k_1, k_2\leq n-1$, $k_1 \neq k_2$,
				we consider the terms of the form 
				\begin{equation*}
					\bar{\gamma}^2, \; \tilde{\gamma}_k\bar{\gamma}, \; \tilde{\gamma}_k^2, \; \tilde{\gamma}_{k_1}\tilde{\gamma}_{k_2}
				\end{equation*}
				and count their occurrences in the summands of (\ref{eq:l=2}).
				In total $n^2$ element of the form $\bar{\gamma}^2$ are produced by the first summand of (\ref{eq:l=2}),
				minus $n(n-1)$ times in the second, $n-1$ in the third and $\binom{n-1}{2}$ in the last.   
				Hence in total
				\begin{equation*}
					n^2-n(n-1)+(n-1)+\binom{n-1}{2}=n+\binom{n-1}{1}+\binom{n-1}{2}=\binom{n}{1}+\binom{n}{2}=\binom{n+1}{2}.
				\end{equation*}
				In total $-2n$ elements of the form $\tilde{\gamma}_k\bar{\gamma}$ are produced in the first summand of (\ref{eq:l=2}),
				$2n-1$ in the second, $-2$ in the third and $2-n$ in the last.
				Hence in total
				\begin{equation*}
					-2n+(2n-1)-2+(2-n)=n+1=\binom{n+1}{1}.
				\end{equation*}
				The terms $\tilde{\gamma}_k^2$ appear once in the first summand of (\ref{eq:l=2}), once in the third and negative once in the second,
				hence once in total.
				The terms $\tilde{\gamma}_{k_1}\tilde{\gamma}_{k_2}$ appear twice in the first summand,
				minus twice in the the second and once in the last,
				hence once in total.
				Therefore the conditions of the proposition are satisfied.
				For $l\geq 3$, we first show that
				\begin{equation*}
					h^{n-l+2}_l=\sum_{\substack{0\leq k \leq l \\ 
					1\leq i_1 \leq \cdots \leq i_k \leq n-l+2}}
					{(-1)^{l-k} \binom{n+1}{l-k} \tilde{\gamma}_{i_1}\cdots\tilde{\gamma}_{i_{k}}\bar{\gamma}^{l-k}}.
				\end{equation*}
				where in the index on the sum here is $1\leq i_1 \leq \cdots \leq i_k \leq n-l+2$ rather than $1\leq i_1 \leq \cdots \leq i_k \leq n-l+1$ in the statement of the proposition.
				The proposition then follows form the statement above by inductions. This is because we have already shown the case $l=2$ and we can obtain the expression for $h^{n-l+2}_l$ in the proposition for $l\geq 3$ from the one above by subtracting off $\tilde{\gamma}_{n-l+2}$ times the expression for $h^{n-l+1}_{l-1}$ in the statement of the proposition,
				as this cancels all the summands containing a multiple $\tilde{\gamma}_{n-l+2}$.
				Using rearrangement (\ref{eq:BaseExpression}),
				\begin{equation}\label{eq:l>3}
					h^{n-l+2}_l=\sum_{1\leq i_1\leq \cdots\leq i_l \leq n-l+2}{\prod_{k=1}^{l}{(\tilde{\gamma}_{i_k}-\bar{\gamma})}}.
				\end{equation}
				For any choice of $1\leq i_1\leq \cdots \leq i_k \leq n-l+2$ and non-negative integers $b,a_1,\dots,a_k$ such that $b+a_1+\cdots+a_k=l$,
				the terms of the form
				\begin{equation}\label{eq:ProdChoice}
					\tilde{\gamma}_{i_1}^{a_1}\cdots\tilde{\gamma}_{i_k}^{a_k}\bar{\gamma}^{b}
				\end{equation}
				describe up to multiplicity all possible summand in the expansion of equation~(\ref{eq:l>3}).
				Define the notation $h^{n-l+2}_l\{ \tilde{\gamma}_{i_1}^{a_1}\cdots\tilde{\gamma}_{i_k}^{a_k}\bar{\gamma}^{b} \}$ to be the multiplicity of the summand containing
				$\tilde{\gamma}_{i_1}^{a_1}\cdots\tilde{\gamma}_{i_k}^{a_k}\bar{\gamma}^{b}$ in the expansion of equation~(\ref{eq:l>3}).
				Using this notation and equation~(\ref{eq:l>3}) if we show that for each $n+1\geq l \geq 3$,
				\begin{equation}\label{eq:BorckekenExpress}
					h^{n-l+2}_l\{ \tilde{\gamma}_{i_1}^{a_1}\cdots\tilde{\gamma}_{i_k}^{a_k}\bar{\gamma}^{b} \} = (-1)^{b} \binom{n+1}{b}
				\end{equation}
				we would complete the proof of the proposition.
				
				Considering each summand of equation~(\ref{eq:l>3}) in tern and counting the number of 
				$\tilde{\gamma}_{i_1}^{a_1}\cdots\tilde{\gamma}_{i_k}^{a_k}\bar{\gamma}^{b}$
				produced in each product, we obtain
				\begin{equation}\label{eq:ChoicExpanssion}
					h^{n-l+2}_l\{ \tilde{\gamma}_{i_1}^{a_1}\cdots\tilde{\gamma}_{i_k}^{a_k}\bar{\gamma}^{b} \} =
					(-1)^b
					\sum_{\theta=0}^{b}{\multiset{n-l+2-k}{b-\theta}\sum_{\substack{\alpha_1+\cdots+\alpha_k=\theta \\ \alpha_j\geq 0}}
					{\prod^{\theta}_{\beta=1}{\binom{a_\beta+\alpha_\beta}{\alpha_\beta}}}}.
				\end{equation}
				We proceed by induction on $n$ and prove (\ref{eq:BorckekenExpress}) for all $n\geq 1$ and $2\leq l\leq n+1$.
				When $n=1$, the only valid value of $l$ is $2$ and $h^{n-l+2}_{l}=(\tilde{\gamma}_1-\bar{\gamma})^2$ whose expansions satisfies (\ref{eq:BorckekenExpress}). 
				Assume that (\ref{eq:BorckekenExpress}) holds for all $\phi\leq n$.
				It is clear that if $b=0$ or $n+1$, then $h^{n-l+2}_{l}\{ \bar{\gamma}^{n+1} \}=(-1)^{n+1}$ and $h^{n-l+2}_{l}\{ \tilde{\gamma}_{i_1}^{a_1}\cdots\tilde{\gamma}_{i_k}^{a_k} \}=1$
				for any choice of $a_1,\dots,a_{k}$ since in the expansion of equation~(\ref{eq:ChoicExpanssion}) there would be only one way to obtain the element.
				For $1\leq b \leq n$, by induction
				\begin{equation}\label{eq:Case(n,b)}
					\binom{n}{b}=(-1)^{b}h^{n-l+1}_{l}\{ \tilde{\gamma}_{i_1}^{a_1}\cdots\tilde{\gamma}_{i_k}^{a_k}\bar{\gamma}^{b} \}
					=\sum_{\theta=0}^{b}{\multiset{n-l+1-k}{b-\theta}\sum_{\substack{\alpha_1+\cdots+\alpha_k=\theta \\ \alpha_j\geq 0}}
					 {\prod^{\theta}_{\beta=1}{\binom{a_\beta+\alpha_\beta}{\alpha_\beta}}}}
				\end{equation}
				and
				\begin{equation}\label{eq:Case(n,b-1)}
					\binom{n}{b-1}=(-1)^{b-1}h^{n-l+2}_{l-1}\{ \tilde{\gamma}_{i_1}^{a_1}\cdots\tilde{\gamma}_{i_k}^{a_k}\bar{\gamma}^{b-1} \}
					=\sum_{\theta=0}^{b-1}{\multiset{n-l+2-k}{b-1-\theta}\sum_{\substack{\alpha_1+\cdots+\alpha_k=\theta \\ \alpha_j\geq 0}}
					 {\prod^{\theta}_{\beta=1}{\binom{a_\beta+\alpha_\beta}{\alpha_\beta}}}}.
				\end{equation}
				For each $0\leq\theta\leq b-1$, the sum of values from (\ref{eq:Case(n,b)}) and (\ref{eq:Case(n,b-1)}) corresponds to the $\theta$ summand in the expression for  
				$h^{n-l+2}_l\{ \tilde{\gamma}_{i_1}^{a_1}\cdots\tilde{\gamma}_{i_k}^{a_k}\bar{\gamma}^{b} \}$
				since the binomial expressions agree and the multiset expression sum to the correct result.
				The only reaming summand in $h^{n-l+2}_l\{ \tilde{\gamma}_{i_1}^{a_1}\cdots\tilde{\gamma}_{i_k}^{a_k}\bar{\gamma}^{b} \}$ is the one corresponding to $\theta=b$.
				However this agrees in (\ref{eq:ChoicExpanssion}) and (\ref{eq:Case(n,b)}) because $\multiset{n-l+2-k}{0}=\multiset{n-l+1-k}{0}=1$, with the binomial parts being identical.
			\end{proof}
			
			\begin{remark}\label{rmk:NewBaisSymGrobner}
			    Note that for a fixed $n$ the set of expressions for $h^{n-l+2}_l$ with $2\leq l \leq n+1$ given in Proposition~\ref{prop:basis} remains a reduced Gr\"obner bases for the ideal they generate with respect to the lexicographic term order on variables $\tilde{\gamma}_1<\cdots<\tilde{\gamma}_n<\bar{\gamma}$.
			\end{remark}
			
			The next theorem shows that the change of basis interacts well with the differentials.
			Recall that, the image of $d^m$ for $m\geq 2$ was determined in Theorem~\ref{thm:allDiff} and shown only to be non-trivial on terms containing $x_m$.
			
			\begin{theorem}\label{thm:InductiveDiff}
			    For $2\leq l \leq n+1$,
			    \begin{equation*}
			        d^{2(l-1)}(x_{2(l-1)})=
			        \sum_{\substack{1 \leq i \leq n}}
					{y_i(\tilde{\gamma}_i-\bar{\gamma})^{l-2}\tilde{\gamma}_i}.
			    \end{equation*}
			    After reduction by lower degree differentials this may be written as
			    \begin{equation*}
			        \sum_{\substack{1 \leq i \leq n}}
					{y_i\tilde{\gamma}_i^{l-1}}.
			    \end{equation*}
			\end{theorem}
			
			\begin{proof}
			    Generalise the notation of Remark~\ref{rmk:TildeBasis} to reflect the form of all differentials on $x_{2(l-1)}$ as follows.
			    For each $1\leq i, \leq n$, write
			    \begin{equation*}
			        \tilde{\gamma}_i^{(j)}=
			        \sum_{\substack{|(c_1,\dots,c_{n})|=j}}
					{(c_i+1)\gamma_1^{c_1}\cdots\gamma_i^{c_i}\cdots\gamma_n^{c_{n}}}.
			    \end{equation*}
			    In particular, we have $\tilde{\gamma}_i^{(1)}=\tilde{\gamma_i}$ and by Theorem~\ref{thm:allDiff}
			    \begin{equation*}
			        d^{2j}(x_{2j})=
			        \sum_{\substack{1 \leq i \leq n}}
			        {y_i\tilde{\gamma}_i^{(j)}}.
			    \end{equation*}
			    We next show that by quotienting out by symmetric polynomials,
			    \begin{equation}\label{eq:TildaRelation}
			        \tilde{\gamma}_i^{(j+1)}=\gamma_i\tilde{\gamma}_i^{(j)}
			    \end{equation}
			    and since by definition
			    \begin{equation*}
			        \gamma_i = \tilde{\gamma}_i-\bar{\gamma}
			    \end{equation*}
			    induction on $j$ then completes the proof of the first part of the theorem.
			    Notice that, using Definition~(\ref{defn:CompleteHomogeneous}) of the complete homogeneous symmetric polynomials
			    \begin{align*}
			        \tilde{\gamma}_i^{(j+1)} & =
			        \sum_{\substack{|(c_1,\dots,c_{n})|=j+1}}
					{(c_i+1)\gamma_1^{c_1}\cdots\gamma_i^{c_i}\cdots\gamma_n^{c_{n}}} \\
					& = h_{j+1} +
					\sum_{\substack{|(c_1,\dots,c_{n})|=j+1,
					\\ c_i \geq 1}}
					{c_i\gamma_1^{c_1}\cdots\gamma_i^{c_i}\cdots\gamma_n^{c_{n}}} \\
					& = h_{j+1} +
					\gamma_i\sum_{\substack{|(c_1,\dots,c_{n})|=j+1}}
					{c_i\gamma_1^{c_1}\cdots\gamma_i^{c_i-1}\cdots\gamma_n^{c_{n}}} \\
					& = h_{j+1} + \gamma_i\tilde{\gamma}_i^{(j)}.
			    \end{align*}
			    This proves (\ref{eq:TildaRelation}).
			    The final statement of the theorem is proved by induction subtracting multiples of $\bar{\gamma}^m$ from the lower degree differentials. 
			\end{proof}
			
	    \section{The integral cohomology of $\Lambda(SU(n+1)/T^n)$}\label{sec:Cohomology}
			
		    In this section we use the results in previous section to deduce structure in the final page of $\{ E^r,d^r \}$, the cohomology Leray-Serre spectral sequence of the evaluations fibration (\ref{eq:evalfib}).
		    This result alone performs part of the computation of the integral cohomology of the free loop space of the complete flag manifold of the special unitary  group $SU(n+1)/T^n$ and is applied when $n=3$ in Section~\ref{sec:LSU4/T3}.
		    We first make the following notation and remark on the choice of basis up to sign.
		    For integers $1\leq t \leq j \leq n$ and $1\leq i_1 < \cdots < i_j \leq n$, write
		    \begin{equation}\label{eq:yHat}
		        \hat{y}_{i_1,\dots,i_j}
		    \end{equation}
		    to denote the product $y_{1}y_{2}\cdots y_{n}$ without the elements $y_{i_1},\dots,y_{i_j}$,
		    while
		    \begin{equation}\label{eq:yHatHat}
		        \hat{y}_{i_1,\dots,\hat{i}_t,\dots,i_j}
		    \end{equation}
		    denotes the product $y_{1}y_{2}\cdots y_{n}$ without the elements $y_{i_1},\dots,y_{i_j}$ except $y_{i_t}$ which is still included.
		    
		    \begin{remark}\label{rmk:ySign}
		        By choosing the appropriate signs on generators $y_i$ for each $1\leq i \leq n$, we express the image of differentials given in Theorem~\ref{thm:InductiveDiff} by
		        \begin{equation*}
		            d^{2l}(x_{2l})=
			        \sum_{\substack{1 \leq i \leq n}}
					{(-1)^{i+1}y_i(\tilde{\gamma}_i-\bar{\gamma})^{l-1}\tilde{\gamma}_i}
		        \end{equation*}
		        where $1\leq l \leq n$.
		        This choice is made to simplify multiplication of the differential by terms $y_1,\dots,y_n$ in the following way.
		        Recall that the generators $y_i$ generate an exterior algebra, in particular $y_i^2=0$.
		        So using Remark \ref{rmk:DifOnDivPoly}, for any $m\geq 1$, $1\leq j,l \leq n$ and $1\leq i_1 < \cdots < i_j \leq n$ we have
		        \begin{align*}
					d^{2l}((x_{2l})_m\hat{y}_{i_1,\dots,i_j})
					&=d^{2l}((x_{2l})_m)\hat{y}_{i_1,\dots,i_j} \\
					&=(x_{2l})_{m-1}\sum_{t=1}^j{(-1)^{i_t+1}(-1)^{i_t+t-2}\hat{y}_{i_1,\dots,\hat{i}_t,\dots,i_j}(\tilde{\gamma}_{i_t}-\bar{\gamma})^{l-1}\tilde{\gamma}_{i_t}} \\
					&=(x_{2l})_{m-1}\sum_{t=1}^j{(-1)^{t-1}\hat{y}_{i_1,\dots,\hat{i}_t,\dots,i_j}
					(\tilde{\gamma}_{i_t}-\bar{\gamma})^{l-1}\tilde{\gamma}_{i_t}}
		        \end{align*}
		        where the additional $(-1)^{i_t+t-2}$ sign changes come from reordering the $y_i$.
				The generator $y_t$ swaps places with $y_i$, $i_t-1$ times for $i<t$ changing the sign each time,
				however $t-1$ of these $y_i$ are missing.
		        Therefore for the remainder of the section we assume the expression for the differential given above.
		    \end{remark}
		    
		    The next theorem deduces for all $n\geq 2$ important structural information on the interaction between the images of differentials and the symmetric polynomials in a special case relevant to the spectral sequence $\{ E^r,\; d^r \}$.
		    In particular the second part of the theorem solves a particular case of part (2) of Proposition~\ref{thm:SpectralGrobner} and the first part describe the torsion generators in this case as detailed at the end of Proposition~\ref{thm:SpectralGrobner}.
		    Hence this result alone perform part of the calculations in Section~\ref{sec:LSU4/T3}.   
			
		    \begin{theorem}\label{thm:Ideals}
		        Given integers
		        \begin{equation*}
		            1\leq i\leq n-1,\; 1\leq j,l \leq n,\; 2\leq j'\leq n+1 ,\; 1\leq k \leq n+1,\; 1\leq t < l,\; 1\leq i_1 \leq i_2 \leq \cdots \leq i_j \leq n
		        \end{equation*}
		        consider all ideals in $\mathbb{Z}[\tilde{\gamma}_i,\bar{\gamma},y_j]$ with $d^{2l}$ image as described in Remark~\ref{rmk:ySign}.
		        Then the following hold:
		        \begin{enumerate}
		        \item
    		        There is a equality of ideals
    		        \begin{equation*}\label{eq:BottemGrober}
    		            \langle d^{2l}(x_{2l}\hat{y}_j),\; y_1\cdots y_n h_{j'}^{n-j'+2} \rangle
    		            =
    		            \langle y_1\cdots y_n\tilde{\gamma}_i,\; \binom{n+1}{k}y_1\cdots y_n\bar{\gamma}^k \rangle. 
    		        \end{equation*}
                \item
    			    For a fixed choice of $2 \leq l \leq n$, the intersection of ideals
    			    \begin{equation*}\label{eq:IntersectBottem}
    			        \langle d^{2l}(x_{2l}\hat{y}_j) \rangle
    			        \cap
    			        \langle d^{2t}(x_{2t}\hat{y}_j),\; y_1\cdots y_n h_{j'}^{n-j'+2} \rangle
    			    \end{equation*}
    			    is given by
    			    \begin{align*}\label{eq:Bottem}
    			        y_1\cdots y_n \langle & \left(\lcm \left(n+1,\binom{n+1}{j'}\right)
    			        /\binom{n+1}{j'}
    			        \right)
    			        h_{j'}^{n-j'+2},
    			        \tilde{\gamma}_i h_{j'}^{n-j'+2}, (n+1)\bar{\gamma}h_{j'}^{n-j'+2} \rangle
    			    \end{align*}
    			    if $l=1$ and
    			    \begin{equation*}
    			        \langle d^{2l}(x_{2l}\hat{y}_j) \rangle
    			    \end{equation*}
    			    otherwise.
			    \end{enumerate}
			\end{theorem}
			
			\begin{proof}
			    We begin by proving part $(2)$ in the case $l=1$, then extend to the general case.
			    Components of the proof of part $(2)$ are then used to prove part $(1)$ in the case $l=1$.
			    First rearranging the generator representatives, then using the second part of Theorem~\ref{thm:InductiveDiff} and Remark~\ref{rmk:TildeBasis}, we see that
			    \begin{align}\label{eq:d^2Ideal}
			        \langle d^{2}(x_2\hat{y}_j) \rangle
			        =& \langle (-1)^{i+1}d^2(x_2)\hat{y}_i,\; (-1)^{n+1}d^2(x_2)(\hat{y}_n+\sum_{1\leq k\leq n-1}(-1)^{i+1}\hat{y}_k)\rangle \nonumber
			        \\
			        =& \langle y_1\dots y_n \tilde{\gamma_i},\; (n+1)y_1\dots y_n \bar{\gamma} \rangle.
			    \end{align}
			    So when $l=1$ generators $y_1\cdots y_n\tilde{\gamma}_i h_{j'}^{n-j'+2}$ and $y_1\cdots y_n(n+1)\bar{\gamma}h_{j'}^{n-j'+2}$ lie on both sides of the intersection of ideals in part $(2)$,
			    hence are contained in the intersection.
			    By Proposition~\ref{prop:basis},
			    \begin{equation}\label{eq:SymBasis}
					y_1,\dots,y_nh^{n-j'+2}_{j'}=y_1,\dots,y_n\sum_{\substack{0\leq t \leq j' \\ 1\leq i_1 \leq \cdots \leq i_t \leq n-j'+1}}
					{(-1)^{j'-t} \binom{n+1}{j'-t} \tilde{\gamma}_{i_1}\cdots\tilde{\gamma}_{i_{t}}\bar{\gamma}^{j'-t}}.
				\end{equation}
				Notice that the terms of the sum with $t>0$,
				are contained in $\langle y_1\dots y_n \tilde{\gamma}_i \rangle$, hence in $\langle d^{2l}(x_l\hat{y}_j) \rangle$ by (\ref{eq:d^2Ideal}).
				Using again (\ref{eq:d^2Ideal}),
				the left hand ideal of the intersection contains
				the generator $(n+1)\bar{\gamma}$.
				The summands of (\ref{eq:SymBasis}) with $t=0$, are divisible by $\bar{\gamma}$ not $\tilde{\gamma}_i$ and so are contained in the intersection
				only when divisible by $n+1$.
				Hence a scalar multiple of $h_{j'}^{n-{j'}+2}$ is in intersection
				when multiplied by the least common multiple of the $n+1$ and $\binom{n+1}{j'}$ divided by $\binom{n+1}{j'}$ since all other terms in the summands of \eqref{eq:SymBasis} for $t>0$ are divisible by a $\tilde{\gamma}_i$.
				Since all necessary degrees have been considered this completes the proof of prat $(2)$ when $l=1$.
				Considering now $2 \leq l \leq n$, by Theorem~\ref{thm:InductiveDiff} combined with Remark~\ref{rmk:ySign}, we have
				\begin{equation*}
			        d^{2l}(x_{2l})=
			        \sum_{\substack{1 \leq i \leq n}}
					{(-1)^{i+1}y_i(\tilde{\gamma}_i-\bar{\gamma})^{l-1}\tilde{\gamma}_i}.
			    \end{equation*}
				Hence applying the discussion in Remark~\ref{rmk:ySign}, we obtain
				\begin{equation*}
			        d^{2l}(x_{2l}\hat{y}_j )=
					{(-1)^{j+1}y_1\cdots y_n(\tilde{\gamma}_j-\bar{\gamma})^{l-1}\tilde{\gamma}_j}.
			    \end{equation*}
				Therefore
				\begin{equation*}
				    \langle d^{2l}(x_{2l}\hat{y}_j),\; y_1\cdots y_n h_{j'}^{n-j'+2} \rangle
				    =
				    \langle d^{2}(x_2\hat{y}_j),\; y_1\cdots y_n h_{j'}^{n-j'+2} \rangle
				\end{equation*}
				so the generators of $\langle \hat{y}_jd^{2l}(x_{2l}),\; y_1\cdots y_n h_{j'+1}^{n-j'+1} \rangle$ when $l>1$ may be omitted when considering the intersection of the ideals in part $(2)$ of the theorem.
				In addition the equality of ideals in part $(1)$ also follows from (\ref{eq:d^2Ideal}) and (\ref{eq:SymBasis}).
				When the generators of (\ref{eq:d^2Ideal}) are extended with those from (\ref{eq:SymBasis}),
				using the discussion below (\ref{eq:SymBasis})
				we see that the sum in (\ref{eq:SymBasis}) can be reduced to just $\binom{n+1}{k}y_1,\dots,y_n\bar{\gamma}$ for each $k=j'>1$.
				This leaves the required set of ideal generators.
			\end{proof}
			
		\begin{remark}\label{rmk:CasePre-image}
		    For the purposes of obtaining pre-images required in part (2) of Proposition~\ref{thm:SpectralGrobner}, we now express the intersection of ideas in part (2) of Theorem~\ref{thm:Ideals} when $2 \leq l \leq n$ in terms of $d^2(x_2)$ multiples free of terms divisible by symmetric polynomials.
		    As discussed in the proof of Theorem~\ref{thm:Ideals} $y_1\cdots y_n\tilde{\gamma}_i$ and $(n+1)y_1\cdots y_n\bar{\gamma}_n$ are in the image of the $d^2$ differential, so $y_1\cdots y_n\tilde{\gamma}_ih^{n-j'+2}_{j'}$ and $(n+1)y_1\cdots y_n\bar{\gamma}_nh^{n-j'+2}_{j'}$ need not be further considered. The remaining generators are
		    \begin{equation}\label{eq:HitGenerators}
		        \left(\lcm\left(n+1,\binom{n+1}{j'}\right)/\binom{n+1}{j'}\right)
		        y_1\cdots y_nh_{j'}^{n-j'+2}
		    \end{equation}
		    for each of which we want to obtain an generator of the preimage under $d^{2}$.
		    
		    Recall that using Remarks~\ref{rmk:TildeBasis},~\ref{rmk:ySign} and Theorem~\ref{thm:InductiveDiff}  we have that
		    \[
		        d^2(x_2\hat{y}_i)=
		        y_1\cdots y_n\tilde{\gamma}_i
		        \;\;\;\ \text{and} \;\;\;
		        d^2(x_2\hat{y}_n)=y_1\cdots y_n((n+1)\bar{\gamma}-\tilde{\gamma}_1-\cdots-\tilde{\gamma}_{n-1}).
		    \]
		    Hence we see that
		    \begin{align*}
		        & d^2 \Bigg{(}\lcm\left(n+1,\binom{n+1}{j'}\right)
		        x_2
		        \Bigg{(} (-1)^{j'}\frac{1}{n+1}
		        \bar{\gamma}^{j'-1}
		        (\hat{y}_1+\cdots+\hat{y}_{n-1}+\hat{y}_n)
		        \\ \nonumber & \;\;\;\; +
		        \frac{1}{\binom{n+1}{j'}}
		        \Bigg(
		        \sum_{\substack{1\leq t \leq j' \\ 1\leq i_1 \leq \cdots \leq i_t \leq n-j'+1}}
				{(-1)^{j'-t} \binom{n+1}{j'-t}\hat{y}_{i_1} \tilde{\gamma}_{i_2}\cdots\tilde{\gamma}_{i_{t}}\bar{\gamma}^{j'-t}}
				\Bigg) \Bigg{)}
				\Bigg{)}
				\\ & =
				\left(\lcm\left(n+1,\binom{n+1}{j'}\right)/\binom{n+1}{j'}\right)
				y_1\cdots y_n
		        \sum_{\substack{0\leq t \leq j' \\ 1\leq i_1 \leq \cdots \leq i_t \leq n-j'+1}}
				{(-1)^{j'-t} \binom{n+1}{j'-t} \tilde{\gamma}_{i_1}\cdots\tilde{\gamma}_{i_{t}}\bar{\gamma}^{j'-t}}
		    \end{align*}
		    which using Proposition~\ref{prop:basis} we can see is the same as the expressions in \eqref{eq:HitGenerators}.
		    Therefore
		    \begin{align}\label{eq:PreImageHitGenerators}
		        &
		        \lcm\left(n+1,\binom{n+1}{j'}\right)
		        \Bigg(\frac{1}{n+1}
		        \bar{\gamma}^{j'-1}
		        (\hat{y}_1+\cdots+\hat{y}_{n-1}+\hat{y}_n)
		        \\ \nonumber & + 
		        \frac{1}{\binom{n+1}{j'}}
		        \Bigg(
		        \sum_{\substack{1\leq t \leq j' \\ 1\leq i_1 \leq \cdots \leq i_t \leq n-j'+1}}
				{(-1)^{t} \binom{n+1}{j'-t}\hat{y}_{i_1} \tilde{\gamma}_{i_2}\cdots\tilde{\gamma}_{i_{t}}\bar{\gamma}^{j'-t}}
				\Bigg)\Bigg)
		    \end{align}
		    are representatives of the preimage under $d^2$ of generators in \eqref{eq:HitGenerators} as required.
		    We remove a multiple of $(-1)^{j'}$ as this only changes the sign of the generators.
		\end{remark}

		\section{The integral cohomology of $\Lambda(SU(4)/T^3)$}\label{sec:LSU4/T3}
		
			Theorem~\ref{thm:Ideals} gives enough information to deduce the algebra structure for the final page of $\{ E^r,d^r \}$ in a special case within the cohomology Leray-Serre spectral sequence of the evaluations fibration (\ref{eq:evalfib}) converging to $H^*(\Lambda(SU(n+1)/T^n);\mathbb{Z})$.
			Using Theorem~\ref{thm:Ideals}, Theorem~\ref{thm:InductiveDiff} and Proposition~\ref{thm:SpectralGrobner} we can deduce the final page completely in the case when $n=3$, which is considerably more complex than the rank $2$ case considered in \cite{Burfitt2018}.
        
        \begin{theorem}\label{thm:CohomologySU4}
			    The $E^\infty$-page of the spectral sequence $\{ E^r,d^r \}$ of the evaluations fibration (\ref{eq:evalfib}) converging to $H^*(\Lambda(SU(4)/T^3);\mathbb{Z})$ is given by the algebra $A/I$, where
			    $A$ is the free graded commutative algebra
			    \begin{align*}
			        A = \Lambda_{\mathbb{Z}} ( &
			        \tilde{\gamma}_{1} ,\; \tilde{\gamma}_{2} ,\; \bar{\gamma} ,\;
			        y_1 ,\; y_2 ,\; y_3
			        ,\;
			        (x_2)_{m_2}(x_4)_{m_4}(x_6)_{m_6}\tilde{\gamma}_2\tilde{\gamma}_1^2\bar{\gamma}^3,
			        (x_2)_{m_2}(x_4)_{m_4}(x_6)_{m_6}y_1y_2y_3
			        ,\; \\ &
			        (x_2)_{m_2}(x_4)_{m_4}(x_6)_{m_6}(\tilde{\gamma}_1^2+\tilde{\gamma}_1\tilde{\gamma}_2+\tilde{\gamma}_2^2+4(\tilde{\gamma}_1+\tilde{\gamma}_2)\bar{\gamma}+6\bar{\gamma}^2)
			        ,\; \\ &
			        (x_2)_{m_2}(x_4)_{m_4}(x_6)_{m_6}
			        (\tilde{\gamma}_1^3+4\tilde{\gamma}_1^2\bar{\gamma}+6\tilde{\gamma}_1\bar{\gamma}^2+4\bar{\gamma}^3)
			        ,\;
			        (x_2)_{m_2}(x_4)_{m_4}(x_6)_{m_6}\bar{\gamma}^4
					,\; \\ &
				    (x_2)_{m_2}(x_4)_{b_4}(x_6)_{b_6}(\hat{y}_1(2\tilde{\gamma}_1+2\tilde{\gamma}_2-5\bar{\gamma})+\hat{y}_2(2\tilde{\gamma}_2-5\bar{\gamma})+3\hat{y}_3\bar{\gamma})
			        ,\; \\ & 
			        (x_2)_{m_2}(x_4)_{b_4}(x_6)_{b_6}(\hat{y}_1(\tilde{\gamma}_1^2-4\tilde{\gamma}_1\bar{\gamma}+5\bar{\gamma}^2)-\hat{y}_1\bar{\gamma}^2-\hat{y}_3\bar{\gamma}^2)
			        ,\;
			        (x_2)_{m_2}(x_4)_{b_4}(x_6)_{b_6}(\hat{y}_1+\hat{y}_2+\hat{y}_3)\bar{\gamma}^2
				    ,\; \\ &
			    	(x_4)_{m_4}(x_6)_{b_6}(\tilde{\gamma}_2\tilde{\gamma}_1^2-6\tilde{\gamma}_2\bar{\gamma}^2-6\tilde{\gamma}_1\bar{\gamma}^2-14\bar{\gamma}^3)
					,\;
					(x_4)_{m_4}(x_6)_{b_6}(\tilde{\gamma}_2\tilde{\gamma}_1\bar{\gamma}+3\tilde{\gamma}_2\bar{\gamma}^2+3\tilde{\gamma}_1\bar{\gamma}^2+6\bar{\gamma}^3)
					,\; \\ &
					(x_4)_{m_4}(x_6)_{b_6}\tilde{\gamma}_2\bar{\gamma}^3
					,\;
					(x_4)_{m_4}(x_6)_{b_6}(\tilde{\gamma}_1^2\bar{\gamma}+\tilde{\gamma}_1\bar{\gamma}^2)
					,\;
					(x_4)_{m_4}(x_6)_{b_6}\tilde{\gamma}_1\bar{\gamma}^3
					,\; \\ &
					(x_6)_{m_6}\tilde{\gamma}_1
					,\;
					(x_6)_{m_6}\tilde{\gamma}_2
					,\;
					(x_6)_{m_6}\bar{\gamma}
					,\; \\ &
					(x_2)_{m_2}(x_4)_{b_4}(x_6)_{b_6}(y_1\tilde{\gamma}_1-y_2\tilde{\gamma}_2-y_3(\tilde{\gamma}_2+\tilde{\gamma}_1+4\bar{\gamma}))
					, \; \\ &
					(x_4)_{m_4}(x_6)_{a_6}y_1\tilde{\gamma}_1^2-y_2\tilde{\gamma}_2^2+y_3(\tilde{\gamma}_2+\tilde{\gamma}_1+4\bar{\gamma})^2
					, \; \\ &
					(x_6)_{m_6}y_1\tilde{\gamma}_1^3-y_2\tilde{\gamma}_2^3-y_3(\tilde{\gamma}_2+\tilde{\gamma}_1+4\bar{\gamma})^3
					, \; \\ &
					(x_2)_{m_2}(x_4)_{b_4}(x_6)_{b_6}(
					y_1(\tilde{\gamma}_2^2+4\tilde{\gamma}_2\bar{\gamma}+6\bar{\gamma}^2)
			        -y_2(\tilde{\gamma}_1^2+4\tilde{\gamma}_1\bar{\gamma}+6\bar{\gamma}^2)
			        -y_3(\tilde{\gamma}_2^2+4\tilde{\gamma}_2\bar{\gamma}-4\bar{\gamma}^2))
			        ,\; \\ &
			        (x_2)_{m_2}(x_4)_{b_4}(x_6)_{b_6}(
			        y_1(\tilde{\gamma}_2\bar{\gamma}^3+\tilde{\gamma}_1\bar{\gamma}^3)
			        +y_2\tilde{\gamma}_1\bar{\gamma}^3
			        -y_3\tilde{\gamma}_2\bar{\gamma}^3
					,\; \\ &
					(x_2)_{m_2}(x_4)_{b_4}(x_6)_{b_6}(
					y_1(2\tilde{\gamma}_2\bar{\gamma}^2+2\tilde{\gamma}_1\bar{\gamma}^2+8\bar{\gamma}^3)
                    -y_2(3\tilde{\gamma}_1^2\bar{\gamma}+10\tilde{\gamma}_1\bar{\gamma}^2+10\bar{\gamma}^3)
                    \\ & \;\;\;\;\;\;\;\;\;\;\;\;
                    +y_3(2\tilde{\gamma}_2\tilde{\gamma}_1^2+8\tilde{\gamma}_2\tilde{\gamma}_1\bar{\gamma}+10\tilde{\gamma}_2\bar{\gamma}^2+5\tilde{\gamma}_1^2\bar{\gamma}+20\tilde{\gamma}_1\bar{\gamma}^2+22\bar{\gamma}^3))
					,\; \\ &
					(x_2)_{m_2}(x_4)_{b_4}(x_6)_{b_6}(
					y_2\tilde{\gamma}_1^2\bar{\gamma}^3
					+y_3\tilde{\gamma}_1^2\bar{\gamma}^3)
					,\; \\ &
					(x_2)_{m_2}(x_4)_{b_4}(x_6)_{b_6}(
					y_2(3\tilde{\gamma}_1^2\bar{\gamma}^2+12\tilde{\gamma}_1\bar{\gamma}^3)
                    \\ & \;\;\;\;\;\;\;\;\;\;\;\;
                    -y_3(2\tilde{\gamma}_2\tilde{\gamma}_1^2\bar{\gamma}+8\tilde{\gamma}_2\tilde{\gamma}_1\bar{\gamma}^2+12\tilde{\gamma}_2\bar{\gamma}^3+5\tilde{\gamma}_1^2\bar{\gamma}^2+20\tilde{\gamma}_1\bar{\gamma}^3))
					,\; \\ &
					(x_2)_{m_2}(x_4)_{b_4}(x_6)_{b_6}
					y_3(\tilde{\gamma}_2\tilde{\gamma}_1^2\bar{\gamma}^2+4\tilde{\gamma}_2\tilde{\gamma}_1\bar{\gamma}^3+4\tilde{\gamma}_1^2\bar{\gamma}^3)
					,\; \\ &
					(x_4)_{m_4}(x_6)_{b_6}(y_1(\tilde{\gamma}_2+3\bar{\gamma})-y_2(\tilde{\gamma}_1+3\bar{\gamma})+y_3(2\tilde{\gamma}_2+2\tilde{\gamma}_1+6\bar{\gamma}))
			        ,\; \\ &
			        (x_4)_{m_4}(x_6)_{b_6}(y_1\tilde{\gamma}_1+y_2(\tilde{\gamma}_1+3\bar{\gamma})-y_3(2\tilde{\gamma}_2+\tilde{\gamma}_1+5\bar{\gamma}))
			        ,\; \\ &
			        (x_4)_{m_4}(x_6)_{b_6}(y_1\bar{\gamma}^2+y_3(\tilde{\gamma}_1^2+3\tilde{\gamma}_1\bar{\gamma}+3\bar{\gamma}^2))
			        ,\; \\ &
			        (x_4)_{m_4}(x_6)_{b_6}(2y_1\bar{\gamma}-y_2(2\tilde{\gamma}_1+4\bar{\gamma})+y_3(2\tilde{\gamma}_2+2\tilde{\gamma}_1+6\bar{\gamma}))
			        ,\; \\ &
			        (x_4)_{m_4}(x_6)_{b_6}(y_2(\tilde{\gamma}_2+\tilde{\gamma}_1+3\bar{\gamma})-y_3(\tilde{\gamma}_2+\bar{\gamma}))
			        ,\; \\ &
			        (x_4)_{m_4}(x_6)_{b_6}(y_2\tilde{\gamma}_1^2-y_3(5\tilde{\gamma}_1^2+12\tilde{\gamma}_1\bar{\gamma}+12\bar{\gamma}^2))
			        ,\; \\ &
			        (x_4)_{m_4}(x_6)_{b_6}(y_2\tilde{\gamma}_1\bar{\gamma}+y_3(3\tilde{\gamma}_1^2+7\tilde{\gamma}_1\bar{\gamma}+6\bar{\gamma}^2))
			        ,\; \\ &
			        (x_4)_{m_4}(x_6)_{b_6}(y_2\bar{\gamma}^2-y_3(\tilde{\gamma}_1^2+2\tilde{\gamma}_1\bar{\gamma}+\bar{\gamma}^2))
			        ,\; \\ &
			        (x_4)_{m_4}(x_6)_{b_6}y_3(\tilde{\gamma}_2\bar{\gamma}+\tilde{\gamma}_1\bar{\gamma}+4\bar{\gamma}^2)
			        ,\; \\ &
			        (x_4)_{m_4}(x_6)_{b_6}y_3(\tilde{\gamma}_1^2\bar{\gamma}-2\bar{\gamma}^3)
			        ,\;
			         (x_4)_{m_4}(x_6)_{b_6}y_3(\tilde{\gamma}_1\bar{\gamma}^2+2\bar{\gamma}^3)
					,\; \\ &
					(x_6)_{m_6}(y_1-y_2)
					,\;
					(x_6)_{m_6}y_3
					)
			    \end{align*}
			    with $|\tilde{\gamma}_{j}|=|\bar{\gamma}|=2$, $|(x_k)_{m_k}|=km_k$ and $|(x_k)_{a_k}|=ka_k$, $I$ is the ideal
			    \begin{align*}
			        I = [ & 
			        (x_2)_{a_2}
			        ((x_2)_1^b-!(x_2)_b)s_2
			        ,\; \\ &
			        (x_4)_{a_4}
			        ((x_4)_1^b-!(x_4)_b)s_4
			        ,\; \\ &
                    (x_6)_{a_6}
			        ((x_6)_1^b-!(x_6)_b)s_6
			        ,\; \\ &
			        (x_2)_{a_2}(x_4)_{a_4}(x_6)_{a_6}(\tilde{\gamma}_1^2+\tilde{\gamma}_1\tilde{\gamma}_2+\tilde{\gamma}_2^2+4(\tilde{\gamma}_1+\tilde{\gamma}_2)\bar{\gamma}+6\bar{\gamma}^2)
			        ,\; \\ &
			        (x_2)_{a_2}(x_4)_{a_4}(x_6)_{a_6}
			        (\tilde{\gamma}_1^3+4\tilde{\gamma}_1^2\bar{\gamma}+6\tilde{\gamma}_1\bar{\gamma}^2+4\bar{\gamma}^3)
			        ,\; \\ &
			        (x_2)_{a_2}(x_4)_{a_4}(x_6)_{a_6}\bar{\gamma}^4
			        ,\; \\ &
			        (x_2)_{a_2}(x_4)_{a_4}(x_6)_{a_6}
			        (y_1\tilde{\gamma}_1+y_2\tilde{\gamma}_2+y_3(4\bar{\gamma}-\tilde{\gamma}_1-\tilde{\gamma}_2))
			        ,\; \\ &
			        (x_2)_{a_2}(x_4)_{a_4}(x_6)_{a_6}
			        (y_1\tilde{\gamma}_1^2+y_2\tilde{\gamma}_2^2+y_3(4\bar{\gamma}-\tilde{\gamma_1}-\tilde{\gamma}_2)^2)
			        ,\; \\ &
			        (x_2)_{a_2}(x_4)_{a_4}(x_6)_{a_6}
			        (y_1\tilde{\gamma}_1^3+y_2\tilde{\gamma}_2^3+y_3(4\bar{\gamma}-\tilde{\gamma}_1-\tilde{\gamma}_2)^3)
			        ]
			    \end{align*}
			    for elements $s_k\in S_k$ given by
			    \begin{equation*}
			        S_k = \{ a/{(x_k)_{m_k}} \; | \; a \text{ is a generator of } A \text{ divisible by some } (x_k)_{m_k} \}
			    \end{equation*}

			    where $b\geq 1$, $k=2,\:4,\:6$ and $m_k,a_k\geq 0$ with the additional condition that some $m_k\geq 1$ when appearing in a generator of $A$.
			    
			    Furthermore, the cohomology algebra $H^*(\Lambda(SU(4)/T^3);\mathbb{Z})$ is isomorphic as a module
			    to the algebra $A/I$ up to order of $2$-torsion and $4$-torsion.
			    In addition there are no multiplicative extension problem on
			    the sub-algebra generated by $\gamma_1, \: \gamma_2, \: \gamma_3$, the sub-algebra generated by $y_1, \: y_2, \: y_3$
			    and no multiplicative extension on elements $y_i\gamma_j$ for $1\leq i,j\leq 3$.
			\end{theorem}
			
			\begin{proof}
			    Throughout the proof all indices lie in the same ranges given in the statement of the theorem.
			    All examples of code used for computation can be found in \cite{Burfitt2021}.
			    Consider the cohomology Leray-Serre spectral sequence $\{ E_r,d^r \}$ associated to the evaluation fibration of $SU(4)/T^3$ studied in Section~\ref{sec:FreeLoopSU(n+1)/Tn},
				\begin{equation*}
					\Omega(SU(4)/T^{3})\to\Lambda(SU(4)/T^3)\to SU(4)/T^3.
				\end{equation*}
				We first consider the algebra structure of the $E_{\infty}$-page of the spectral sequence,
				then consider the implications for the cohomology algebra.
                By Theorem~\ref{thm:H*SU/T} and using the basis of Remark~\ref{rmk:TildeBasis}, the integral cohomology of the base space $SU(4)/T^3$ is given by
			    \begin{equation*}
				    \frac{\mathbb{Z}[\tilde{\gamma}_1,\:\tilde{\gamma}_2,\:\bar{\gamma}]}{\langle h_{2}^{3},\: h_{3}^{2},\: h_{4}^{1} \rangle}
			    \end{equation*}
			    where $|\tilde{\gamma}_j|=|\bar{\gamma}|=2$.
			    Using Proposition~\ref{prop:basis} after changing the sign of the $\bar{\gamma}$ generator, we have
			    \begin{align}\label{eq:SU4SymetricQuotient}
			        & h_{2}^{3} = 
			        \tilde{\gamma}_2^2+\tilde{\gamma}_1\tilde{\gamma}_2+\tilde{\gamma}_1^2+4(\tilde{\gamma}_2+\tilde{\gamma}_1)\bar{\gamma}+6\bar{\gamma}^2
			        ,\; \nonumber \\ &
			        h_{3}^{2} =
			        \tilde{\gamma}_1^3+4\tilde{\gamma}_1^2\bar{\gamma}+6\tilde{\gamma}_1\bar{\gamma}^2+4\bar{\gamma}^3
			        \;  \\ \text{and } &
			        h_{4}^{1} = \bar{\gamma}^4. \nonumber
			    \end{align}
			    Therefore, the maximal degree of elements of the algebra is $6$.
			    From~(\ref{eq:BaseLoopFlag}), the integral cohomology of the fibre $\Omega(SU(4)/T^3)$ is given by
			    \begin{equation}\label{eq:SU4LoopFibre}
				    \Lambda_\mathbb{Z}(y_1,\:y_2,\:y_3)\otimes\Gamma_\mathbb{Z}[x_2,\:x_4,\:x_6]
			    \end{equation}
			    where $|y_i|=1$ and $|x_k|=k$.
			    Applying Theorem~\ref{thm:GrobnerOver} with the Gr\"obner basis from Theorems \ref{thm:monomial sum}, the additive generators on the $E_2$-page of the spectral sequence are given by representative elements of the form
			    \begin{equation}\label{eq:SU4Representatives}
				    (x_2)_{a_2}(x_4)_{a_4}(x_6)_{a_6}y_{\alpha_1}\cdots y_{\alpha_l}P
			    \end{equation}
			    where $0\leq l \leq 3$, $1\leq \alpha_1< \cdots < \alpha_{l}\leq 3$ and $P\in \mathbb{Z}[\tilde{\gamma}_1,\:\tilde{\gamma}_2,\:\bar{\gamma}]$ is a monomial of degree at most $6$.
			    
			    By Theorem~\ref{thm:allDiff}, the only non-zero differentials are $d^2,\;d^4$ and $d^6$
			    which are non-zero only on the generators $x_2,\:x_4$ and $x_6$, respectively.
			    Therefore the spectral sequence converges by the seventh page.
			    Using Theorem~\ref{thm:InductiveDiff} and substituting $\tilde{\gamma}_3$ for $-(4\bar{\gamma}+\gamma_1+\gamma_2)$ in the basis of Remark~\ref{rmk:TildeBasis} and the sign change on $\bar{\gamma}$ made in (\ref{eq:SU4SymetricQuotient}), the images of the differentials up to sign are generated by
			    \begin{align}\label{eq:SU4Difs}
			        d^2(x_2) &= y_1\tilde{\gamma}_1-y_2\tilde{\gamma}_2-y_3(4\bar{\gamma}+\gamma_1+\gamma_2)
			        , \nonumber \\
			        d^4(x_4) &= y_1\tilde{\gamma}_1^2-y_2\tilde{\gamma}_2^2+y_3(4\bar{\gamma}+\tilde{\gamma}_1+\tilde{\gamma}_2)^2
			         \\
			        \text{and } d^6(x_6) &=
			        y_1\tilde{\gamma}_1^3-y_2\tilde{\gamma}_2^3-y_3(4\bar{\gamma}+\gamma_1+\gamma_2)^3.
			        \nonumber
			    \end{align}
			    At this point we immediately obtain a number of the generators and relations occurring in $A$ and $I$ of the statement of the theorem.
			    The monomial generators
			    \begin{equation*}
    			    \tilde{\gamma}_1,\:\tilde{\gamma}_2,\:\bar{\gamma},\: y_1,\:y_2,\:y_3,
    			 (x_2)_{m_2}(x_4)_{m_4}(x_6)_{m_6}\tilde{\gamma}_2\tilde{\gamma}_1^2\bar{\gamma}^3
    			 \text{ and }
    			 (x_2)_{m_2}(x_4)_{m_4}(x_6)_{m_6}y_1y_2y_3
			    \end{equation*}
			    occur in $E_2^{*,0}$ or $E_2^{0,*}$ and are always in the kernel of the differentials, so are algebra generators of the $E_\infty$-page.
				All relations on the $E_7$-page coming from the divided polynomial relations in
				$H^*(\Omega(SU(4)/T^3);\mathbb{Z})$ given in (\ref{eq:SU4LoopFibre}) of the form
                $(x_k)_1^m-m!(x_k)_m$,
                the symmetric relations in (\ref{eq:SU4SymetricQuotient}) and
                images of the differentials (\ref{eq:SU4Difs})
				hold on the $E_\infty$-page and therefore are in $I$.
				It is also necessary to include the multiple of these relations by $(x_2)_{a_2}(x_4)_{a_4}(x_6)_{a_6}$
				to ensure that they occur as generators of the algebra $A$.
			    In addition, we add to the relations $((x_k)_1^m-m!(x_k)_m)s_k$, where the multiple of an element $s_k$ from the set
				$S_k$ ensures all generators that occur as a multiple of some $(x_k)_{a_k}$ appear in $I$.
			    
			    We have considered all possible relations occurring on the $E_{\infty}$-page of the spectral sequence,
			    so it remains to determine all generators of $A$.
			    Any additional generators arise as elements whose image under the differentials obtained using (\ref{eq:SU4Difs}) lie inside the ideal generated by the symmetric relations~(\ref{eq:SU4SymetricQuotient}).
			    To this end we need only consider elements whose image is generated by elements from (\ref{eq:SU4Representatives}) with $l=1,2,3$.
			
			    The case $l=3$ coincides with Theorem~\ref{thm:Ideals}.
			    Since up to $(x_2)_{a_2}(x_4)_{a_4}(x_6)_{a_6}$ multiples, the image of $d^2$ in this case can be rearranged as
			    \begin{equation*}
			        y_1y_2y_2\tilde{\gamma}_1,
			        \; y_1y_2y_2\tilde{\gamma}_2,
			        \; 4y_1y_2y_2\bar{\gamma}
			    \end{equation*}
			    there are no generators corresponding to part (1) of Proposition~\ref{thm:SpectralGrobner} in this case.
			    In addition in the course of the proof of Theorem~\ref{thm:Ideals}
			    it is shown that the images of $d^4$ and $d^6$ in this case are contained in the image of $d^2$ so their domain on these rows lies entirely in the kernel.
			    Part (2) of Proposition~\ref{thm:SpectralGrobner} coincides with the results of part (2) of Theorem~\ref{thm:Ideals}, 
			    so using the expression \eqref{eq:PreImageHitGenerators} from Remark \ref{rmk:CasePre-image} we add to the generator of $A$, the following expressions 
			    \begin{equation*}
			        \hat{y}_1(2\tilde{\gamma}_1+2\tilde{\gamma}_2-5\bar{\gamma})+\hat{y}_2(2\tilde{\gamma}_2-5\bar{\gamma})+3\hat{y}_3\bar{\gamma}
			        ,\;
			        \hat{y}_1(\tilde{\gamma}_1^2-4\tilde{\gamma}_1\bar{\gamma}+5\bar{\gamma}^2)-\hat{y}_1\bar{\gamma}^2-\hat{y}_3\bar{\gamma}^2
			        ,\;
			        (\hat{y}_1+\hat{y}_2+\hat{y}_3)\bar{\gamma}^2
			        .
			    \end{equation*}
			
			    As noted in Remark~\ref{rmk:NewBaisSymGrobner}, the generators (\ref{eq:SU4SymetricQuotient}) are a Gr{\"o}bner basis with respect to the lexicographic minimal ordering given by
			    \begin{equation*}
			        \tilde{\gamma}_2 > \tilde{\gamma}_1 > \bar{\gamma}.
			    \end{equation*}
			    Hence using Theorem~\ref{thm:GrobnerOver}, we see that any element
			    on the $E_2$-page
			    may be reduced by elements of (\ref{eq:SU4SymetricQuotient}) to a unique form and this is zero if and only if it is a multiple of an element of the symmetric ideal. In addition, from the leading terms of (\ref{eq:SU4SymetricQuotient}) it is clear that we need only consider up to
			    \begin{equation*}
			        \tilde{\gamma}_2\tilde{\gamma}_1^2\bar{\gamma}^3
			    \end{equation*}
			    multiples of the image of the differentials given in (\ref{eq:SU4Difs}).
			    
			    We now consider elements in the rows of the spectral sequence corresponding to elements in (\ref{eq:SU4Representatives}) when $l=1$.
			    In this case the image of $d^2,d^4$ and $d^6$ on $x_2,x_4$ and $x_6$, respectively are generated by a single element, so there are no generators of the kernel corresponding to part (1) of Proposition~\ref{thm:SpectralGrobner} that we have not already included.
			    Applying part (2) of Proposition~\ref{thm:SpectralGrobner} to rows in this case,
			    we make all Gr\"obner bases computations up to degree $6$ in variables $\tilde{\gamma}_1,\: \tilde{\gamma}_2\:, \bar{\gamma}$ and degree $1$ in variables $y_1,\:y_2,\:y_3$.
			    For the general case it is sufficient to consider the (\ref{eq:SU4SymetricQuotient}) reduced forms of $\tilde{\gamma}_1,\:\tilde{\gamma}_2,\:\bar{\gamma}$ multiples of the images of differentials given in (\ref{eq:SU4Difs}), as other rows will be multiples of theses by elements of $\Gamma_{\mathbb{Z}[x_2.x_4,x_6]}$.
			    In computer computations, we use the extend lexicographic monomial ordering
			    \begin{equation*}
			        y_1 > y_2 > y_3 > \tilde{\gamma}_2 > \tilde{\gamma}_1 > \bar{\gamma}.
			    \end{equation*}
			    In the case of $\phi^2_2$ on multipels of $x_2$, using Gr\"obner bases to compute the intersection of the ideals generated by $\phi^2_2(x_2)$ and (\ref{eq:SU4SymetricQuotient}) gives an ideal generated by the three elements $\phi^2_2(x_2)h^3_2$, $\phi^2_2(x_2)h^2_3$ and $\phi^2_2(x_2)h^4_1$,
			    all of which are are multiples of elements of (\ref{eq:SU4SymetricQuotient}), hence no additional generators need be added to $A$ in this case. 

                Computing the Gr\"obner bases to obtain the intersection of the ideals generated by $\phi^4_4$ on multiples of $x_4$ and (\ref{eq:SU4SymetricQuotient}) gives an ideal generated by elements $\phi^4_4(x_4)h^2_3$, $\phi^4_4(x_4)h^4_1$ and
                \begin{align*}
                    & \phi^4_4(x_4)(\tilde{\gamma}_2\tilde{\gamma}_1^2-6\tilde{\gamma}_2\bar{\gamma}^2-6\tilde{\gamma}_1\bar{\gamma}^2-14\bar{\gamma}^3), \\
                    & \phi^4_4(x_4)(\tilde{\gamma}_2\tilde{\gamma}_1\bar{\gamma}+3\tilde{\gamma}_2\bar{\gamma}^2+3\tilde{\gamma}_1\bar{\gamma}^2+6\bar{\gamma}^3), \\
                    & \phi^4_4(x_4)\tilde{\gamma}_2\bar{\gamma}^3, \\
                    & \phi^4_4(x_4)(\tilde{\gamma}_1^3+2\tilde{\gamma}_1\bar{\gamma}^2+4\bar{\gamma}^3), \\
                    & \phi^4_4(x_4)(\tilde{\gamma}_1^2\bar{\gamma}+\tilde{\gamma}_1\bar{\gamma}^2), \\
                    & \phi^4_4(x_4)\tilde{\gamma}_1\bar{\gamma}^3.
                \end{align*}
                Since
                \begin{equation*}
                    h_2^3 = (\tilde{\gamma}_1^3+2\tilde{\gamma}_1\bar{\gamma}^2+4\bar{\gamma}^3)+4(\tilde{\gamma}_1^2\bar{\gamma}+\tilde{\gamma}_1\bar{\gamma}^2)
                \end{equation*}
                we need not add any generators to $A$ as multiples $x_4(\tilde{\gamma}_1^3+2\tilde{\gamma}_1\bar{\gamma}^2+4\bar{\gamma}^3)$.
			    It can again be checked by Gr\"obner bases applying part (2) of Proposition~\ref{thm:SpectralGrobner} that $(x_4)_{m_4}(x_6)_{b_6}$ multiples of
			    \begin{align*}
			        \tilde{\gamma}_2\tilde{\gamma}_1^2-6\tilde{\gamma}_2\bar{\gamma}^2-6\tilde{\gamma}_1\bar{\gamma}^2-14\bar{\gamma}^3,\:
			        \tilde{\gamma}_2\tilde{\gamma}_1\bar{\gamma}+3\tilde{\gamma}_2\bar{\gamma}^2+3\tilde{\gamma}_1\bar{\gamma}^2+6\bar{\gamma}^3,\:
			        \tilde{\gamma}_2\bar{\gamma}^3,\:
			        \tilde{\gamma}_1^2\bar{\gamma}+\tilde{\gamma}_1\bar{\gamma}^2
			        \text{ and } \tilde{\gamma}_1\bar{\gamma}^3
			    \end{align*}
			    are in the kernel of $d^6$, therefore they are added to $A$ as generators since the spectral sequence converges on the $E_7$-page.
			    Here part (1) of Proposition~\ref{thm:SpectralGrobner} need not be considered as we have already confirmed that everything lies in the kernel.

			    Finally the Gr\"obner bases of the intersection of the ideals generated by $\phi^6_6(x_6)$ and (\ref{eq:SU4SymetricQuotient}) gives an ideal generated by elements
			    \begin{equation*}
			        d^6(x_6)\tilde{\gamma}_2,\: d^6(x_6)\tilde{\gamma}_1,\: d^6(x_6)\bar{\gamma}
			    \end{equation*}
			    all of whose $(x_6)_{m_6}$  multiples are added to $A$ as generators.
			    
			    Lastly we consider elements in the rows of the spectral sequence corresponding to the elements in (\ref{eq:SU4Representatives}) when $l=2$.
			    In the case of $\phi^2_2$ on multiples of $y_ix_2$ for $i=1,2,3$
			    having image generated by
			    \begin{align*}
			        y_1d^2(x_2) & = -y_1y_2\tilde{\gamma}_2 -y_1y_3(\tilde{\gamma}_2+\tilde{\gamma}_1+4\bar{\gamma}),
			        \\
			        y_2d^2(x_2) & = -y_1y_2\tilde{\gamma}_1 -y_2y_3(\tilde{\gamma}_2+\tilde{\gamma}_1+4\bar{\gamma})
			        \\
			        \text{and } y_3d^2(x_2) & = -y_1y_3\tilde{\gamma}_1 +y_2y_3\tilde{\gamma}_2.
			    \end{align*}
			    We first note that
			    \begin{equation*}
			        y_1d^2(x_2)\tilde{\gamma}_1-y_2d^2(x_2)\tilde{\gamma}_2-y_3d^2(x_2)(\tilde{\gamma}_2+\tilde{\gamma}_1+4\bar{\gamma}) = 0.
			    \end{equation*}
			    Similarly we have the following relations of the $d^4$ and $d^6$ differentials,
			    \begin{align*}
			        & y_1d^2(x_2)\tilde{\gamma}_1^2-y_2d^2(x_2)\tilde{\gamma}_2^2+y_3d^2(x_2)(\tilde{\gamma}_2+\tilde{\gamma}_1+4\bar{\gamma})^2 = 0
			        \\ \text{and} \;\;\; & y_1d^2(x_2)\tilde{\gamma}_1^3-y_2d^2(x_2)\tilde{\gamma}_2^3-y_3d^2(x_2)(\tilde{\gamma}_2+\tilde{\gamma}_1+4\bar{\gamma})^3 = 0.
			    \end{align*}
			    Therefore 
			    $(x_2)_{m_2}(x_4)_{b_4}(x_6)_{b_6}$  multiples of $y_1\tilde{\gamma}_1-y_2\tilde{\gamma}_2-y_3(\tilde{\gamma}_2+\tilde{\gamma}_1+4\bar{\gamma})$,
			    $(x_4)_{m_4}(x_6)_{b_6}$  multiples of $y_1\tilde{\gamma}_1^2-y_2\tilde{\gamma}_2^2+y_3(\tilde{\gamma}_2+\tilde{\gamma}_1+4\bar{\gamma})^2$
			    and
			    $(x_6)_{m_6}$  multiples of $y_1\tilde{\gamma}_1^3-y_2\tilde{\gamma}_2^3-y_3(\tilde{\gamma}_2+\tilde{\gamma}_1+4\bar{\gamma})^3$
			    are added to $A$.
			    It can be checked by computing the Syzygys of the ideals spanned by the image of the three sets of differentials that these are the only additional generators not already included in $A$ corresponding to part (1) of Proposition~\ref{thm:SpectralGrobner}.
			    Considering part (2) of Proposition~\ref{thm:SpectralGrobner}, using Gr\"obner bases to compute the intersection of the ideals generated by $\phi^2_2y_i(x_2)$ and (\ref{eq:SU4SymetricQuotient}) gives an ideal generated by elements corresponding to
			    $y_i\phi^2_2(x_2)h_3^2,\: y_i\phi^2_2(x_2)h_2^3,\: y_i\phi^2_2(x_2)h_1^4$ and
			    \begin{align*}
			    	& -y_1\phi^2_2(x_2)(\tilde{\gamma}_2^2+4\tilde{\gamma}_2\bar{\gamma}+6\bar{\gamma}^2)
			        +y_2\phi^2_2(x_2)(\tilde{\gamma}_1^2+4\tilde{\gamma}_1\bar{\gamma}+6\bar{\gamma}^2)
			        +y_3\phi^2_2(x_2)(\tilde{\gamma}_2^2+4\tilde{\gamma}_1^2
			        -4\bar{\gamma}^2)
			        ,\\
			        & -y_1\phi^2_2(x_2)(\tilde{\gamma}_2\bar{\gamma}^3+\tilde{\gamma}_1\bar{\gamma}^3)
			        -y_2\phi^2_2(x_2)\tilde{\gamma}_1\bar{\gamma}^3
			        +y_3\phi^2_2(x_2)\tilde{\gamma}_2\bar{\gamma}^3
			        ,\\
			        & -y_1\phi^2_2(x_2)(2\tilde{\gamma}_2\bar{\gamma}^2+2\tilde{\gamma}_1\bar{\gamma}^2+8\bar{\gamma}^3)
			        +y_2\phi^2_2(x_2)(3\tilde{\gamma}_1^2\bar{\gamma}+10\tilde{\gamma}_1\bar{\gamma}^2+10\bar{\gamma}^3) \\ & \;\;\;\;\;\;\;\;\;\;\;\;
			        -y_3\phi^2_2(x_2)(2\tilde{\gamma}_2\tilde{\gamma}_1^2+8\tilde{\gamma}_2\tilde{\gamma}_1\bar{\gamma}+10\tilde{\gamma}_2\bar{\gamma}^2+5\tilde{\gamma}_1^2\bar{\gamma}+20\tilde{\gamma}_1\bar{\gamma}^2+22\bar{\gamma}^3)
			        ,\\
			        & -y_2\phi^2_2(x_2)\tilde{\gamma}_1^2\bar{\gamma}^3-y_3\phi^2_2(x_2)\tilde{\gamma}_1^2\bar{\gamma}^3
			        ,\\
			        & -y_2\phi^2_2(x_2)(3\tilde{\gamma}_1^2\bar{\gamma}^2+12\tilde{\gamma}_1\bar{\gamma}^3)
			        +y_3\phi^2_2(x_2)(2\tilde{\gamma}_2\tilde{\gamma}_1^2\bar{\gamma}+8\tilde{\gamma}_2\tilde{\gamma}_1\bar{\gamma}^2+12\tilde{\gamma}_2\bar{\gamma}^3+5\tilde{\gamma}_1^2\bar{\gamma}^2+20\tilde{\gamma}_1\bar{\gamma}^3)
			        ,\\
			        & -y_3\phi^2_2(x_2)(\tilde{\gamma}_2\tilde{\gamma}_1^2\bar{\gamma}^2+4\tilde{\gamma}_2\tilde{\gamma}_1\bar{\gamma}^3+4\tilde{\gamma}_1^2\bar{\gamma}^3)
			    \end{align*}
			    all of whose corresponding $d^2$ preimage $(x_2)_{m_2}(x_4)_{b_4}(x_6)_{b_6}$  multiples are added to $A$ as generators up to sign as it can again be check with Gr\"obner bases that all generators remain in the kernels of the $d^4$ and $d^6$ differentials.
			    
			    Using Gr\"obner bases to compute the intersection of the ideals generated by $y_i\phi^4_4(x_4)$ and (\ref{eq:SU4SymetricQuotient}) gives an ideal generated by elements corresponding to
			    $y_3\phi^4_4(x_4)h_1^4$ and
			    \begin{align}
			        & \label{eq:kerd4row21}
			        -y_1\phi^4_4(x_4)(\tilde{\gamma}_2+3\bar{\gamma})+y_2\phi^4_4(x_4)(\tilde{\gamma}_1+3\bar{\gamma})-y_3\phi^4_4(x_4)(2\tilde{\gamma}_2+2\tilde{\gamma}_1+6\bar{\gamma})
			        ,\\ & \label{eq:kerd4row22}
			        -y_1\phi^4_4(x_4)\tilde{\gamma}_1-y_2\phi^4_4(x_4)(\tilde{\gamma}_1+3\bar{\gamma})+y_3\phi^4_4(x_4)(2\tilde{\gamma}_2+\tilde{\gamma}_1+5\bar{\gamma})
			        ,\\ & \label{eq:kerd4row23}
			        -y_1\phi^4_4(x_4)\bar{\gamma}^2-y_3\phi^4_4(x_4)(\tilde{\gamma}_1^2+3\tilde{\gamma}_1\bar{\gamma}+3\bar{\gamma}^2)
			        ,\\ & \label{eq:kerd4row24}
			        -2y_1\phi^4_4(x_4)\bar{\gamma}+y_2\phi^4_4(x_4)(2\tilde{\gamma}_1+4\bar{\gamma})-y_3\phi^4_4(x_4)(2\tilde{\gamma}_2+2\tilde{\gamma}_1+6\bar{\gamma})
			        ,\\ & \label{eq:kerd4row25}
			        -y_2\phi^4_4(x_4)(\tilde{\gamma}_2+\tilde{\gamma}_1+3\bar{\gamma})+y_3\phi^4_4(x_4)(\tilde{\gamma}_2+\bar{\gamma})
			        ,\\ & \label{eq:kerd4row26}
			        -y_2\phi^4_4(x_4)\tilde{\gamma}_1^2+y_3\phi^4_4(x_4)(5\tilde{\gamma}_1^2+12\tilde{\gamma}_1\bar{\gamma}+12\bar{\gamma}^2)
			        ,\\ & \label{eq:kerd4row27}
			        -y_2\phi^4_4(x_4)\tilde{\gamma}_1\bar{\gamma}-y_3\phi^4_4(x_4)(3\tilde{\gamma}_1^2+7\tilde{\gamma}_1\bar{\gamma}+6\bar{\gamma}^2)
			        ,\\ & \label{eq:kerd4row28}
			        -y_2\phi^4_4(x_4)\bar{\gamma}^2+y_3\phi^4_4(x_4)(\tilde{\gamma}_1^2+2\tilde{\gamma}_1\bar{\gamma}+\bar{\gamma}^2)
			        ,\\ & \label{eq:kerd4row29}
			        -y_3\phi^4_4(x_4)(\tilde{\gamma}_2^2-\tilde{\gamma}_1^2-5\tilde{\gamma}_1\bar{\gamma}-10\bar{\gamma}^2)
			        ,\\ & \label{eq:kerd4row210}
			        -y_3\phi^4_4(x_4)(\tilde{\gamma}_2\tilde{\gamma}_1+2\tilde{\gamma}_1^2+5\tilde{\gamma}_1\bar{\gamma})
			        ,\\ & \label{eq:kerd4row211}
			        -y_3\phi^4_4(x_4)(\tilde{\gamma}_2\bar{\gamma}+\tilde{\gamma}_1\bar{\gamma}+4\bar{\gamma}^2)
			        ,\\ & \label{eq:kerd4row212}
			        -y_3\phi^4_4(x_4)\tilde{\gamma}_1^3
			        ,\\ & \label{eq:kerd4row213}
			        -y_3\phi^4_4(x_4)(\tilde{\gamma}_1^2\bar{\gamma}-2\bar{\gamma}^3)
			        ,\\ & \label{eq:kerd4row214}
			         -y_3\phi^4_4(x_4)(\tilde{\gamma}_1\bar{\gamma}^2+2\bar{\gamma}^3).
			    \end{align}
			    Notice that
			    \begin{align*}
			        -y_2\phi^4_4(x_4)h^2_3 = &
			        \tilde{\gamma}_2(\ref{eq:kerd4row25})+(\ref{eq:kerd4row26})+\bar{\gamma}(\ref{eq:kerd4row25})+3(\ref{eq:kerd4row27})+3(\ref{eq:kerd4row28})+(\ref{eq:kerd4row29})+2(\ref{eq:kerd4row211})
			        , \\ 
			        -y_3\phi^4_4(x_4)h^2_3 = &
			        (\ref{eq:kerd4row29})+(\ref{eq:kerd4row210})+4(\ref{eq:kerd4row211})
			        , \\ 
			        -y_3\phi^4_4(x_4)h^3_2 = &
			        (\ref{eq:kerd4row212})+4(\ref{eq:kerd4row213})+6(\ref{eq:kerd4row214}).
			    \end{align*}
			    Hence as can again be checked that all of these generator are in the kernel of $d^6$, the $(x_4)_{m_4}(x_6)_{b_6}$ multiples of $d^4$ preimages of all except
			    (\ref{eq:kerd4row29}), (\ref{eq:kerd4row210}) and (\ref{eq:kerd4row212})
			    are added to $A$ as generators up to sign.
			    
			    Using Gr\"obner bases to compute the intersection of the ideals generated by $\phi^6_6y_i(x_6)$ and (\ref{eq:SU4SymetricQuotient}) gives an ideal generated by the elements corresponding to
			    \begin{align*}
			        & 
			        -y_1\phi^6_6(x_6)+y_2\phi^6_6y_i(x_6)
			        ,\;
			        -y_2\phi^6_6(x_6)\tilde{\gamma}_2
			        ,\;
			        -y_2\phi^6_6(x_6)\tilde{\gamma}_1
			        ,\;
			        -y_2\phi^6_6(x_6)\bar{\gamma}
			        \text{ and }
			        -y_3\phi^6_6(x_6).
			    \end{align*}
			    Only the first and last of whose $(x_6)_{m_6}$ multiples are added to $A$ as generators up to sign, as the others are already products of exiting generators.
			    
			    By computing the Gr\"obner bases corresponding to (\ref{eq:GrobnerTorsionPart}) of Proposition~\ref{thm:SpectralGrobner}, we determine that only $2$-torsion and $4$-torsion occurs on the $E_\infty$-page of the spectral sequence.
			    Hence module structure of the integral cohomology algebra of $SU(4)/T^3$ up to torsion type is determined by looking at the spectral sequence in modulo $2$ coefficients in the same way as in \cite[Theorem~5.2]{Burfitt2018} where the modulo $3$ spectral sequence is considered.
			    The only remaining additive extension problem is whether the $4$-torsion on the $E_\infty$-page is $2$-torsion or $4$-torsion in $H^*(\Lambda (SU(4)/T^3);\mathbb{Z})$.
			    The multiplicative extension problems on certain subalgebras are also determined in the same way as in the proof of \cite[Theorem~5.2]{Burfitt2018}.
			\end{proof}
        
\bibliographystyle{amsplain}
\bibliography{Ref}
			
\end{document}